\documentclass{article}%
\usepackage{amsfonts}
\usepackage{amsmath}
\usepackage{amssymb}
\usepackage{graphicx}
\usepackage[pdfstartview=FitH]{hyperref}%
\providecommand{\U}[1]{\protect\rule{.1in}{.1in}}
\newtheorem{theorem}{Theorem}

\newtheorem{conjecture}[theorem]{Conjecture}
\newtheorem{corollary}[theorem]{Corollary}
\newtheorem{example}[theorem]{Example}
\newtheorem{lemma}[theorem]{Lemma}

\newtheorem{remark}[theorem]{Remark}
\newenvironment{proof}[1][Proof]{\noindent\textbf{#1.} }{\ \rule{0.5em}{0.5em}}
\newenvironment{definition}[1][Definition]{\noindent\textbf{#1} }{}
\newenvironment{acknowledgement}[1][Acknowledgement]{\noindent\textbf{#1} }{}
\begin{document}

\title{Subsingular vectors in Verma modules, and tensor product of weight modules
over the twisted Heisenberg-Virasoro algebra and $W(2,2)$ algebra}
\author{Gordan Radobolja}
\maketitle

\begin{abstract}
We show that subsingular vectors may exist in Verma modules over $W(2,2)$, and
present the subquotient structure of these modules. We prove conditions for
irreducibility of the tensor product of intermediate series module with a
highest weight module. Relation to intertwining operators over vertex operator
algebra associated to $W(2,2)$ is discussed. Also, we study the tensor product
of intermediate series and a highest weight module over the twisted
Heisenberg-Virasoro algebra, and present series of irreducible modules with
infinite-dimensional weight spaces.

\end{abstract}

\section{Introduction}

Lie algebra $W(2,2)$ was first introduced by W.\ Zhang and C.\ Dong in
\cite{Zhang-Dong} as a part of classification of simple vertex operator
algebras generated by two weight two vectors. It is an extension of the well
known Virasoro algebra $\operatorname*{Vir}$, and its representation theory is
somewhat similar to that of $\operatorname*{Vir}$. Criterion for
irreducibility of Verma modules over $W(2,2)$ was given in \cite{Zhang-Dong}.
The structure of those modules was discussed by W.\ Jiang and Y.\ Pei in
\cite{Jiang-Pei}. However, the authors overlooked an important fact - that a
submodule generated by a singular vector is not necessarily isomorphic to some
Verma module; and therefore missed an interesting possibility - a subsingular
vector in Verma module. In Section \ref{Verma} we present new results on
structure of Verma modules (Theorem \ref{glavni}), and formulas for
subsingular vector. Necessary condition for the existence of a subsingular
vector is given, and many examples supporting a conjecture that this condition
is sufficient are shown.

It was proved by D.\ Liu and L.\ Zhu in \cite{Liu-Zhu} that every irreducible
weight module over $W(2,2)$ with finite-dimensional weight subspaces is either
the highest or lowest weight module, or a module belonging to an intermediate
series. Modules with infinite-dimensional weight subspaces over the affine
Kac-Moody algebra (see \cite{Chari-Pressley} and
\cite{Adamovic1,Adamovic2,Adamovic3}) and over the Virasoro algebra
(\cite{Zhang} and \cite{Radobolja2}) have been studied recently, motivated by
their connection with theory of vertex operator algebras (VOAs) and fusion
rules in conformal field theory. Based on modules studied in these papers, we
consider the tensor product of irreducible module from intermediate series,
and irreducible highest weight module. We show in Section \ref{VOA} how these
modules can be obtained from intertwining operators for modules over VOA
associated to $W(2,2)$. In Section \ref{irr tezor} we classify irreducible
tensor products (Theorem \ref{main}). The existence of, and a formula for a
subsingular vector is crucial in this analysis since generic singular vectors
are of no use (Theorem \ref{skroz red}). We show that a tensor product module
contains an irreducible submodule with infinite-dimensional weight subspaces
if and only if a subsingular vector exists in the corresponding Verma module
(Corollary \ref{irr sub}). Different highest weight modules occur as
subquotients in reducible tensor product modules. Some of these subquotients
are related to intertwining operators.

In section \ref{HV} we discuss irreducibility of tensor product module over
the twisted Heisenberg-Virasoro algebra at level zero. This algebra was
studied by Y.\ Billig in \cite{Billig}, as it appears in the construction of
modules for the toroidal Lie algebras (see \cite{Billig2}). We present a rich
series of irreducible tensor products.

\section{Lie algebra $W(2,2)$}

$W(2,2)$ is a complex Lie algebra with basis $\{W_{n},L_{n},C,:n\in
\mathbb{Z\}}$, and a Lie bracket
\begin{gather}
\left[  L_{n},L_{m}\right]  =\left(  n-m\right)  L_{n+m}+\delta_{n,-m}%
\frac{n^{3}-n}{12}C,\label{1}\\
\left[  L_{n},W_{m}\right]  =\left(  n-m\right)  W_{n+m}+\delta_{n,-m}%
\frac{n^{3}-n}{12}C,\nonumber\\
\left[  W_{n},W_{m}\right]  =\left[  \mathcal{L},C\right]  =0.\nonumber
\end{gather}
In this paper we write $\mathcal{L}=W(2,2)$, a notation used by Liu and Zhu in
\cite{Liu-Zhu}. Obviously, $\{L_{n},C:n\in\mathbb{Z\}}$ spans a copy of the
Virasoro algebra. Triangular decomposition is given by $\mathcal{L}%
=\mathcal{L}_{-}\oplus\mathcal{L}_{0}\oplus\mathcal{L}_{+}$ where
\begin{align*}
\mathcal{L}_{+} &  =\bigoplus\limits_{n>0}(\mathbb{C}L_{n}+\mathbb{C}W_{n}),\\
\mathcal{L}_{-} &  =\bigoplus\limits_{n>0}(\mathbb{C}L_{-n}+\mathbb{C}%
W_{-n}),\\
\mathcal{L}_{0} &  =\mathbb{C}L_{0}+\mathbb{C}W_{0}+\mathbb{C}C.
\end{align*}
However, $W_{0}$ does not act semisimply on the rest of $\mathcal{L}$. Algebra
$\mathcal{L}$ is $\mathbb{Z}$-graded by eigenvalues of $L_{0}$, namely
$\mathcal{L}_{n}=\mathbb{C}L_{-n}+\mathbb{C}W_{-n}+\delta_{n,0}\mathbb{C}C$.

Let $U(\mathcal{L})$ denote a universal enveloping algebra and $\mathcal{I}$ a
left ideal generated by $\{L_{n},W_{n},C-c\mathbf{1},L_{0}-h\mathbf{1}%
,W_{0}-h_{W}\mathbf{1}:n\in\mathbb{N}\}$ for fixed $c,h,h_{W}\in\mathbb{C}$.
Then $V(c,h,h_{W}):=U(\mathcal{L})/\mathcal{I}$ is called the Verma module
with central charge $c$ and highest weight $(h,h_{W})$ or simply with highest
weight $(c,h,h_{W})$. It is a free $U(\mathcal{L})$-module generated by the
highest weight vector $v:=\mathbf{1}+\mathcal{I}$, and a standard
Poincare-Birkhoff-Witt (PBW) basis
\[
\{W_{-m_{s}}\cdots W_{-m_{1}}L_{-n_{t}}\cdots L_{-n_{1}}v:m_{s}\geq\cdots\geq
m_{1}\geq1,n_{t}\geq\cdots\geq n_{1}\geq1\}.
\]
\textbf{Throughout the rest of this paper }$v$\textbf{\ denotes the highest
weight vector and we assume that a monomial }%
\[
W_{-m_{s}}\cdots W_{-m_{1}}L_{-n_{t}}\cdots L_{-n_{1}}v
\]
\textbf{is given in PBW\ basis, unless otherwise noted.} Central charge and
weights should be clear from the context, as well as weather $v$ comes from
the Verma module or some quotient.

Module $V(c,h,h_{W})$ admits a natural gradation by $L_{0}$-eigenspaces i.e.,
weight subspaces: $V(c,h,h_{W})=\bigoplus_{k\geq0}V(c,h,h_{W})_{h+k}$. PBW
vectors such that $\sum m_{i}+\sum n_{j}=k$ form a basis for $V(c,h,h_{W}%
)_{h+k}$. Module $V(c,h,h_{W})$ has a unique maximal submodule $J(c,h,h_{W})$
and $L(c,h,h_{W})=V(c,h,h_{W})/J(c,h,h_{W})$ is unique (up to isomorphism)
irreducible highest weight module with highest weight $(c,h,h_{W})$.

\begin{theorem}
[\cite{Zhang-Dong}]\label{zd}Verma module $V(c,h,h_{W})$ is irreducible if and
only if $2h_{W}+\frac{m^{2}-1}{12}c\neq0$ for any $m\in\mathbb{N}$.
\end{theorem}

\bigskip

Intermediate series $\mathcal{L}$-modules are intermediate series
$\operatorname*{Vir}$-modules with trivial action of $W_{n}$. For
$\alpha,\beta\in\mathbb{C}$, define $V_{\alpha,\beta,0}=\operatorname*{span}%
_{\mathbb{C}}\{v_{m}:m\in\mathbb{Z}\}$ with
\[
L_{n}v_{m}=-(m+\alpha+\beta+n\beta)v_{m+n},\quad Cv_{m}=W_{n}v_{m}=0,
\]
for $n,m\in\mathbb{Z}.$ Then $V_{\alpha,\beta,0}\cong V_{\alpha+k,\beta,0}$
for $k\in\mathbb{Z}$, so when $\alpha\in\mathbb{Z}$ we may assume $\alpha=0$.
Module $V_{\alpha,\beta,0}$ is reducible if and only if $\alpha\in\mathbb{Z}$
and $\beta\in\{0,1\}$. Define $V_{0,0,0}^{\prime}:=V_{0,0,0}/\mathbb{C}v_{0}$,
$V_{0,1,0}^{\prime}:=\bigoplus\limits_{m\neq-1}\mathbb{C}v_{m}\subseteq
V_{0,1,0}$ and $V_{\alpha,\beta,0}^{\prime}=V_{\alpha,\beta,0}$ otherwise.
Then $V_{\alpha,\beta,0}^{\prime}$ are all irreducible modules from
intermediate series. See \cite{Liu-Zhu}.

\begin{theorem}
[\cite{Liu-Zhu}]An irreducible weight $\mathcal{L}$-module with
finite-dimen\-sional weight spaces is isomorphic either to a highest (or
lowest) weight module or to $V_{\alpha,\beta,0}^{\prime}$ for some
$\alpha,\beta\in\mathbb{C}$.
\end{theorem}

Tensor product $V_{\alpha,\beta,0}^{\prime}\otimes L(c,h,h_{W})$ carries an
$\mathcal{L}$-module structure with action
\[
a(v_{n}\otimes x):=av_{n}\otimes x+v_{n}\otimes ax\text{, for any }%
a\in\mathcal{L}\text{,}\ x\in L(c,h,h_{W}).
\]
Note that $W_{m}(v_{n}\otimes v)=v_{n}\otimes W_{m}v$. Like in
$\operatorname*{Vir}$ case $\{v_{n}\otimes v:n\in\mathbb{Z}\}$ generates
$V_{\alpha,\beta,0}^{\prime}\otimes L(c,h,h_{W})$ (see (\ref{hh}) in Lemma
\ref{teh}). Moreover, all weight subspaces are infinite-dimensional:
\[
\left(  V_{\alpha,\beta,0}^{\prime}\otimes L\left(  c,h,h_{W}\right)  \right)
_{h+m-\alpha-\beta}=\bigoplus\limits_{n\in\mathbb{Z}_{+}}\mathbb{C}%
v_{n-m}\otimes L\left(  c,h,h_{W}\right)  _{h+n}%
\]

\section{Verma module structure\label{Verma}}

\textbf{In this section we assume }$2h_{W}+\frac{p^{2}-1}{12}c=0$\textbf{\ for
some }$p\in\mathbb{N}$\textbf{.}

Some results of this section are similar to those presented in
\cite{Jiang-Pei} and motivated by \cite{Billig}. However, since $W_{0}$ does
not act semisimply in general (unlike to $I_{0}$ in the Heisenberg-Virasoro
algebra), submodules generated by some singular vectors in Verma modules are
not isomorphic to Verma modules so the maximal submodule $J(c,h,h_{W})$ is not
necessarily cyclic on a singular vector. In fact, we prove the existence of
subsingular vectors in some Verma modules. Therefore, Corollary 3.6 and later
results in \cite{Jiang-Pei} are not correct in general.

First introduce $W$-degree on $\mathcal{L}_{-}$:
\[
\deg_{W}L_{-n}=0,\text{\quad}\deg_{W}W_{-n}=1,
\]
which induces $\mathbb{Z}$-grading on $U(\mathcal{L})$ and on $V(c,h,h_{W})$
\[
\deg_{W}W_{-m_{s}}\cdots W_{-m_{1}}L_{-n_{t}}\cdots L_{-n_{1}}v=s.
\]
Obviously, this grading depends on a basis. Denote with $V_{(k)}^{W}%
(c,h,h_{W})$ the homogeneous components respective to $W$-degree in a PBW\ basis.

For a nonzero $x\in V(c,h,h_{W})$ we denote by $\overline{x}$ its lowest
nonzero homogeneous component with respect to $W$-degree (in a standard PBW
basis). If $x\in V_{(k)}^{W}(c,h,h_{W})$ and $n>0$, then
\begin{align}
W_{n}x  &  \in V(c,h,h_{W})_{(k)}^{W}\oplus V(c,h,h_{W})_{(k+1)}^{W}%
\label{w1}\\
L_{n}x  &  \in V(c,h,h_{W})_{(k-1)}^{W}\oplus V(c,h,h_{W})_{(k)}^{W}.\nonumber
\end{align}

We can define $L$-degree likewise:
\[
\deg_{L}L_{-n}=1,\text{\quad}\deg_{L}W_{-n}=0.
\]
Let
\[
\mathcal{W}=V_{(0)}^{L}(c,h,h_{W})=\left\{  W_{-m_{s}}\cdots W_{-m_{1}}%
v:m_{s}\geq\cdots\geq m_{1}>0\right\}  .
\]

Now we state the main result of this section, a structure theorem for Verma
modules over the Lie algebra $W(2,2)$.

\begin{theorem}
\label{glavni} Let $2h_{W}+\frac{p^{2}-1}{12}c=0$ for some $p\in\mathbb{N}$.
Then there exists a singular vector $u^{\prime}\in V(c,h,h_{W})_{h+p}%
\cap\mathcal{W}$ such that $\overline{u^{\prime}}=W_{-p}v$ and $U(\mathcal{L}%
)u^{\prime}\cong V(c,h+p,h_{W})$. Images of vectors $W_{-m_{s}}\cdots
W_{-m_{1}}L_{-n_{t}}\cdots L_{-n_{1}}v$ such that $m_{i}\neq p$ form a PBW
basis for $L^{\prime}(c,h,h_{W}):=V(c,h,h_{W})/U(\mathcal{L})u^{\prime}$. Moreover:

\begin{description}
\item[(i)] Let $h\neq h_{W}+\frac{(13p+1)(p-1)}{12}+\frac{(1-r)p}{2}$ for all
$r\in\mathbb{N}$. Then $U(\mathcal{L})u^{\prime}=J(c,h,h_{W})$ is a maximal
submodule in $V(c,h,h_{W})$ and $L^{\prime}(c,h,h_{W})=L(c,h,h_{W})$ is irreducible.

\item[(ii)] Suppose $L^{\prime}(c,h,h_{W})$ is reducible. Then there exists a
subsingular vector $u\in V(c,h,h_{W})$ such that $\overline{u}=L_{-p}^{r}v$
for some $r\in\mathbb{N}$. Vectors $u$ and $u^{\prime}$ generate a maximal
submodule $J(c,h,h_{W})$ and the images of all vectors $W_{-m_{s}}\cdots
W_{-m_{1}}L_{-n_{t}}\cdots L_{-n_{1}}v$ in which neither $W_{-p}$ nor
$L_{-p}^{k}$ for $k\geq r$ occur as factors form a PBW basis for irreducible
module $L(c,h,h_{W})=V(c,h,h_{W})/J(c,h,h_{W})$.
\end{description}
\end{theorem}

We suspect $L^{\prime}(c,h,h_{W})$ is reducible if and only if $h=h_{W}%
+\frac{(13p+1)(p-1)}{12}+\frac{(1-r)p}{2}$. The proof of Theorem \ref{glavni}
will be presented in a series of lemmas and theorems in the rest of this
section. We will follow the Billig's \cite{Billig} idea, also used in
\cite{Jiang-Pei}.

The following lemma will often be used throughout this section (see also
analogous result in \cite{Billig}).

\begin{lemma}
[\cite{Jiang-Pei}]\label{key}Let $0\neq x\in V(c,h,h_{W})$ and $\deg
_{W}\overline{x}=k$.

\begin{description}
\item[a)] If $\overline{x}\notin\mathcal{W}$ and $n\in\mathbb{N}$ is the
smallest, such that $L_{-n}$ occurs as a factor in one of the terms in
$\overline{x}$, then the part of $W_{n}x$ of the $W$-degree $k$ is given by
\[
n\left(  2h_{W}+\frac{n^{2}-1}{12}c\right)  \frac{\partial\overline{x}%
}{\partial L_{-n}}.
\]

\item[b)] If $\overline{x}\in\mathcal{W}$, $\overline{x}\notin\mathbb{C}v$ and
$m\in\mathbb{N}\ $is maximal, such that $W_{-m}$ occurs as a factor in one of
the terms of $\overline{x}$, then the part of $L_{m}x$ of the $W$-degree $k-1$
is given by
\[
m\left(  2h_{W}+\frac{m^{2}-1}{12}c\right)  \frac{\partial\overline{x}%
}{\partial W_{-m}}.
\]

\end{description}
\end{lemma}

As a first application of this lemma we have the following result:

\begin{lemma}
[\cite{Jiang-Pei}]\label{w}Let $2h_{W}+\frac{p^{2}-1}{12}c=0$. Then there is a
singular vector $s\in V(c,h,h_{W})_{h+p}$ such that $\overline{s}=W_{-p}v$ or
$\overline{s}=L_{-p}v$.
\end{lemma}

For $x \in V(c,h,h_{W})$ we define $x_{n}$ inductively. Let $x_{0}=x$ and
$x_{n}=x_{n-1}-\overline{x_{n-1}}$. Then
\[
x=\sum_{n\geq0}\overline{x_{n}}%
\]
is a decomposition of $x$ by $W$-degree, i.e., $\deg_{W}\overline{x_{i}}%
<\deg_{W}\overline{x_{j}}$ if $i<j$. Moreover, $\overline{x_{n}}$ is
homogeneous respective to $W$-degree (all components of $\overline{x_{n}}$
have the same $W$-degree). For example
\[
x=\underbrace{L_{-3}v}_{\overline{x_{0}}}+\overbrace{\underbrace{W_{-3}%
v+W_{-1}L_{-2}v}_{\overline{x_{1}}}+\overbrace{\underbrace{W_{-1}^{3}v}%
}_{\overline{x_{2}}}^{x_{2}}}^{x_{1}}.
\]

\begin{theorem}
\label{u'}Let $2h_{W}+\frac{p^{2}-1}{12}c=0$. Then there is a singular vector
$u^{\prime}\in V(c,h,h_{W})_{h+p}\cap\mathcal{W}$ such that $\overline
{u^{\prime}}=W_{-p}v$. Moreover, $U(\mathcal{L})u^{\prime}$ is isomorphic to
the Verma module $V(c,h+p,h_{W})$.
\end{theorem}

\begin{proof}
We know from Lemma \ref{w}, that there exists a singular vector $s$ in
$V(c,h,h_{W})_{h+p}$. If $\overline{s}=W_{-p}v$ we set $u^{\prime}=s$. Suppose
that $\overline{s}=L_{-p}v$. Then we set $u^{\prime}=\frac{1}{p} \left(
W_{0}s-h_{W}s \right)  $. Obviously $u^{\prime}$ is a singular vector such
that $\overline{u^{\prime}}=W_{-p}v$. Factor $\frac{1}{p}$ is due to
normalization by $W_{-p}$. Note that this means there are two linearly
independent singular vectors in $V(c,h,h_{W})_{h+p}$.

Next we show that $u^{\prime}\in\mathcal{W}$. Let $n\in\mathbb{N}$ the
smallest such that $\deg_{L}\overline{u_{n}^{\prime}}>0$ and let $\deg
_{W}\overline{u_{n}^{\prime}}=k$. Moreover, let $m\in\mathbb{N}$ the smallest
such that $L_{-m}$ occurs as a factor in $\overline{u_{n}^{\prime}}$.
Obviously, $m<p$. Then by Lemma \ref{key} a) the part of $W_{m}u_{n}^{\prime}
$ of the $W$-degree $k$ is nonzero. However, $W_{m}u_{n}^{\prime}%
=W_{m}u^{\prime}-W_{m}\overline{u_{n-1}^{\prime}}=0$ because $u^{\prime}$ is a
singular vector, and $\overline{u_{n-1}}\in\mathcal{W}$ by minimality of $n$,
so we got a contradiction and conclude that $\deg_{L}u^{\prime}=0$.

Finally, since $u^{\prime}\in\mathcal{W}$ we have $W_{0}u^{\prime}%
=h_{W}u^{\prime}$ so $U\left(  \mathcal{L}\right)  u$ is isomorphic to the
Verma module with the highest weight $(c,h+p,h_{W})$.
\end{proof}

\begin{example}
$u^{\prime}=W_{-1}v$ is a singular vector in the Verma module $V(c,h,0)$;

$u^{\prime}=(W_{-2}-\frac{3}{4h_{W}}W_{-1}^{2})v$ is a singular vector in the
Verma module $V(c,h,-c/8)$;

$u^{\prime}=(W_{-3}-\frac{2}{h_{W}}W_{-2}W_{-1}+\frac{1}{h_{W}^{2}}W_{-1}%
^{3})v$ is a singular vector in the Verma module $V(c,h,-c/3)$.
\end{example}

\bigskip

\textbf{From now on }$u^{\prime}$\textbf{\ denotes the singular vector from
Theorem \ref{u'}.} We also use the following notation:
\begin{align*}
J^{\prime}(c,h,h_{W})  &  :=U\left(  \mathcal{L}\right)  u^{\prime},\\
L^{\prime}(c,h,h_{W})  &  =V(c,h,h_{W})/J^{\prime}(c,h,h_{W}).
\end{align*}

Let $P_{2}\left(  n\right)  =\sum_{i=0}^{n}P(n-i)P(i)$ where $P$ is a
partition function with $P(0)=1$. Then
\[
\operatorname*{char}V(c,h,h_{W})=q^{h}\sum_{n\geq0}P_{2}(n)q^{n}=q^{h}%
\prod\limits_{k\geq1}(1-q^{k})^{-2}.
\]
From Theorem \ref{u'} we get
\begin{align*}
\operatorname*{char}J^{\prime}(c,h,h_{W})  &  =q^{h+p}\sum_{n\geq0}%
P_{2}(n)q^{n}=q^{h+p}\prod\limits_{k\geq1}(1-q^{k})^{-2},\\
\operatorname*{char}L^{\prime}(c,h,h_{W})  &  =\operatorname*{char}%
V-\operatorname*{char}J^{\prime}=q^{h}(1-q^{p})\sum_{n\geq0}P_{2}(n)q^{n}=\\
&  =q^{h}(1-q^{p})\prod\limits_{k\geq1}(1-q^{k})^{-2}.
\end{align*}

\begin{lemma}
\label{nj}Let $0\neq x\in J^{\prime}(c,h,h_{W})$. Then there exist terms in
$\overline{x}$ containing factor $W_{-p}$.
\end{lemma}

\begin{proof}
We may write $x=yu^{\prime}$, where $y\in U(\mathcal{L}_{-})$. Since
$U(\mathcal{L}_{-})$ has no zero divisors, we get $\overline{x}=\overline
{y}\overline{u^{\prime}}$. However, $\overline{u^{\prime}}=W_{-p}v$ so every
component in $\overline{x}$ with maximal length contains $W_{-p}$.
\end{proof}

\begin{lemma}
\label{baza}Denote by $B^{\prime}$ set of PBW vectors $W_{-m_{s}}\cdots
W_{-m_{1}}L_{-n_{t}}\cdots L_{-n_{1}}v$ modulo $J^{\prime}(c,h,h_{W})$, such
that $m_{i}\neq p$ for $i=1,\ldots,s$. Then $B^{\prime}$ is a basis for
$L^{\prime}(c,h,h_{W})$.
\end{lemma}

\begin{proof}
Since $W_{-p}$ cannot occur in a linear combination of vectors from
$B^{\prime}$ it follows from Lemma \ref{nj} that $B^{\prime}$ is linearly
independent. Simple combinatorics shows that a character of a vector space
spanned by $B^{\prime}$ equals $q^{h}(1-q^{p})\sum_{n\geq0}P_{2}(n)q^{n}$,
which proves $B^{\prime}$ is a basis for $L^{\prime}(c,h,h_{W})$.
\end{proof}

\bigskip If there is a singular vector $s\in V(c,h,h_{W})$ such that
$\overline{s}=L_{-p}v$ (recall Proposition \ref{w}), then $L^{\prime
}(c,h,h_{W})$ is reducible and the image of $s$ is a singular vector in
$L^{\prime}(c,h,h_{W})$. Next we show that $L^{\prime}(c,h,h_{W})$ may contain
a singular vector of weight $h+pr$ for any $r\in\mathbb{N}$, i.e.,
$V(c,h,h_{W})$ may contain a subsingular vector. The singular vector $s$ in
Lemma \ref{w} is just a special case $r=1$.

\begin{theorem}
\label{w irr}If $L^{\prime}(c,h,h_{W})$ is reducible then there is a singular
vector $u\in L^{\prime}(c,h,h_{W})$ such that $\overline{u}=L_{-p}^{r}v$ for
some $r\in\mathbb{N}$.
\end{theorem}

\begin{proof}
If the highest weight module $L^{\prime}(c,h,h_{W})$ is reducible, it must
contain some singular vector not in $\mathbb{C}v$. Let $u\in L^{\prime
}(c,h,h_{W})\setminus\mathbb{C}v$ be a homogeneous vector such that
$\mathcal{L}_{+}u=0$ (i.e., $\mathcal{L}_{+}u\in J^{\prime}(c,h,h_{W})$ if we
consider $u $ as a vector in $V(c,h,h_{W})$). We may assume that $u$ is a
linear combination of vectors from $B^{\prime}$. Note that $B^{\prime}$ is
closed under partial derivations $\frac{\partial}{\partial L_{-n}}$ and
$\frac{\partial}{\partial W_{-n}}$.

Suppose $\overline{u}\in\mathcal{W}$ and $\overline{u}\notin\mathbb{C}v$.
Applying Lemma \ref{key} b) we find $m\in\mathbb{N}$, $m\neq p$, such that
$\overline{L_{m}u}\neq0$ and $W_{-p}$ does not occur in $\overline{L_{m}u}$.
By Lemma \ref{nj}, we get $L_{m}u\notin J^{\prime}(c,h,h_{W})$, a contradiction.

Therefore we may assume $\overline{u}\notin\mathcal{W}$. Let $\deg
_{W}\overline{u}=k$ and let $n$ the smallest such that $L_{-n}$ occurs in
$\overline{u}$. By Lemma \ref{key} a), component of $W_{n}u$ of the $W$-degree
$k$ is $n\left(  2h_{W}+\frac{n^{2}-1}{12}c\right)  \frac{\partial\overline
{u}}{\partial L_{-n}}$. If $n\neq p$ we get contradiction again so $n=p$. Let
$n^{\prime}>p$ the smallest such that $L_{-n^{\prime}}$ occurs as a factor in
$\overline{u}$. Since
\[
W_{n^{\prime}}u=\sum_{i=1}^{t}W_{-m_{k}}\cdots W_{-m_{1}}L_{-n_{t}}%
\cdots\left[  W_{n^{\prime}},L_{-n_{i}}\right]  \cdots L_{-n_{1}}v
\]
and $n_{1}=p$ or $n_{1}\geq n^{\prime}$ the only component of $\overline{u}$
that produces part of $W$-degree $k$ is one in which $L_{-n^{\prime}}$ occurs.
That component equals $n^{\prime}\left(  2h_{W}+\frac{n^{\prime2}-1}%
{12}c\right)  \frac{\partial\overline{u}}{\partial L_{-n^{\prime}}}$. Since
$n^{\prime}\neq p$ and since $W_{-p}$ does not occur as a factor in
$\overline{u}$, we get $0\neq W_{n^{\prime}}u\in J^{\prime}(c,h,h_{W})$ such
that $W_{-p}$ does not occur in $\overline{W_{n^{\prime}}u}$. Again, this is a
contradiction with Lemma \ref{nj}.

So far we have proven that if $L_{-t}$ occurs as a factor in $\overline{u}$,
then $t=p$. It is left to prove that $\overline{u}\in\mathbb{C}L_{-p}^{r}v$,
i.e., that $\deg_{W}\overline{u}=0$. Suppose that $\deg_{W}\overline{u}=k$,
and let $m\in\mathbb{N}$ the highest such that $W_{-m}$ occurs as a factor in
$\overline{u}$. It is easy to see that part of $L_{m}u$ of $W$-degree $k-1 $
equals $m\left(  2h_{W}+\frac{m^{2}-1}{12}c\right)  \frac{\partial\overline
{u}}{\partial W_{-m}}$. Since by Lemma \ref{nj} $m\neq p$, we conclude that
$0\neq L_{m}u\in J^{\prime}(c,h,h_{W})$ such that $W_{-p}$ does not occur as a
factor in $\overline{L_{m}u}$, a contradiction. This proves that $\overline
{u}$ equals $L_{-p}^{r}v$ up to a scalar factor, so we set $\overline
{u}=L_{-p}^{r}v$.
\end{proof}

Let us show that such vectors actually exist.

\begin{example}
$u=L_{-1}v$ is a subsingular vector in the Verma module $V(c,0,0)$;

$u=(L_{-1}^{2}+\frac{6}{c}W_{-2})v$ is a subsingular vector in the Verma
module $V(c,-\frac{1}{2},0)$.
\end{example}

\begin{remark}
\label{uni}If it exists, a singular vector $u$ in $L^{\prime}(c,h,h_{W}%
)_{h+rp}$ is obviously unique up to a scalar factor. When considered as a
vector in $V(c,h,h_{W})_{h+rp}$, we usually take the one representative of
$u+U(\mathcal{L})u^{\prime}$ in which $W_{-p}$ does not occur. Therefore we
may think of $u$ as a unique subsingular vector in the Verma module
$V(c,h,h_{W})$ such that $\overline{u}=L_{-p}^{r}v$.
\end{remark}

\bigskip

Our next goal is to find a necessary condition for the existence of a
subsingular vector. We still assume that $2h_{W}+\frac{p^{2}-1}{12}c=0$.
Suppose that a subsingular vector $u$ exists in $V(c,h,h_{W})$, such that
$\overline{u}=L_{-p}^{r}v$. Consider $L_{p}u\in J^{\prime}(c,h,h_{W})$.
Obviously, a component
\[
L_{p}L_{-p}^{r}v=\left(  n\frac{p^{3}-p}{12}c+2p\sum_{i=1}^{r-1}(h+ip)\right)
L_{-p}^{r-1}v=rp(2h-2h_{W}+(r-1)p)L_{-p}^{r-1}v
\]
occurs in $L_{p}u$. Since $L_{p}u\in J^{\prime}(c,h,h_{W})$, by Lemma \ref{nj}
the coefficient with $L_{-p}^{r-1}v$ in $L_{p}u$ has to be zero. The only
components of $u$ that contribute to $L_{-p}^{r-1}v$ in $L_{p}u$ are
$\lambda_{i}W_{-i}L_{-p}^{r-1}L_{-(p-i)}v$ for $i=1,\ldots,p-1$, and their
contribution is
\[
\lambda_{i}(p+i)(p-i)\left(  2h_{W}+\frac{(p-i)^{2}-1}{12}c\right)
L_{-p}^{r-1}v=2\lambda_{i}h_{W}i(2p-i)\frac{p^{2}-i^{2}}{p^{2}-1}L_{-p}%
^{r-1}v.
\]
It is left to find $\lambda_{i}$. Consider $L_{i}u$ for $i=1,\ldots,p-1$.
Vector $\lambda_{i}L_{i}W_{-i}L_{-p}^{r-1}L_{-(p-i)}v$ produces a component
\[
2\lambda_{i}h_{W}i\frac{p^{2}-i^{2}}{p^{2}-1}L_{-p}^{r-1}L_{-(p-i)}v.
\]
The only other component contributing to $L_{-p}^{r-1}L_{-(p-i)}v$ in $L_{i}u
$ is $L_{-p}^{r}v$, with a coefficient $n(p+i)L_{-p}^{r-1}L_{-(p-i)}v$. Since
$L_{-p}^{r-1}L_{-(p-i)}v$ can not occur in a vector from $J^{\prime}%
(c,h,h_{W})$, the associated coefficient needs to be zero, which gives
$\lambda_{i}=-r\frac{p^{2}-1}{2h_{w}i(p-i)}$. Now we have a formula for the
coefficient of $L_{-p}^{r-1}$:
\[
rp(2h-2h_{W}+(r-1)p)-2rh_{W}\sum_{i=1}^{p-1}i(2p-i)\frac{p^{2}-i^{2}}{p^{2}%
-1}\frac{p^{2}-1}{2h_{w}i(p-i)}=0
\]
leading to
\[
h=h_{W}+\frac{(13p+1)(p-1)}{12}+\frac{(1-r)p}{2}.
\]
This proves the following

\begin{theorem}
\label{dov}Let $2h_{W}+\frac{p^{2}-1}{12}c=0$. If $V(c,h,h_{W})$ contains a
subsingular vector $u$ such that $\overline{u}=L_{-p}^{r}v$ for some
$r\in\mathbb{N}$ then $h=h_{W}+\frac{(13p+1)(p-1)}{12}+\frac{(1-r)p}{2}$.
\end{theorem}

From Theorems \ref{w irr} and \ref{dov} and Lemma \ref{baza} we get the following:

\begin{corollary}
\label{wkrit}Suppose $2h_{W}+\frac{p^{2}-1}{12}c=0$. If $h\neq h_{W}%
+\frac{(13p+1)(p-1)}{12}+\frac{(1-r)p}{2}$ for all $r\in\mathbb{N}$, then
$J^{\prime}(c,h,h_{W})=J(c,h,h_{W})$ is a maximal submodule in $V(c,h,h_{W}%
)$and module $L^{\prime}(c,h,h_{W})=L(c,h,h_{W})$ is irreducible with PBW
basis
\[
B^{\prime}=\left\{  W_{-m_{s}}\cdots W_{-m_{1}}L_{-n_{t}}\cdots L_{-n_{1}%
}v:m_{j}\neq p\right\}
\]
and character
\[
\operatorname*{char}L^{\prime}(c,h,h_{W})=q^{h}(1-q^{p})\sum_{n\geq0}%
P_{2}(n)q^{n}=q^{h}(1-q^{p})\prod\limits_{k\geq1}(1-q^{k})^{-2}.
\]

\end{corollary}

\bigskip

Assume now that there is a subsingular vector $u\in V(c,h,h_{W})$ such that
$\overline{u}=L_{-p}^{r}v$. We show that $U(\mathcal{L})\left\{  u,u^{\prime
}\right\}  =J(c,h,h_{W})$ is a maximal submodule in $V(c,h,h_{W})$.

First we need a generalization of Lemma \ref{key} to $L^{\prime}(c,h,h_{W})$.

\begin{lemma}
\label{sub}Let $x=W_{-m_{k}}\cdots W_{-m_{1}}L_{-n_{t}}\cdots L_{-n_{1}}v\in
L^{\prime}(c,h,h_{W})$ such that $m_{j}=p$ for some $j$, and $m_{i}\neq p$ for
$i\neq j$. Then $\deg_{W}\overline{x}\geq k$, if $x$ is considered in a basis
$B^{\prime}$.
\end{lemma}

\begin{proof}
We make use of relation $u^{\prime}=0$ in $L^{\prime}(c,h,h_{W})$. This leads
to $W_{-p}v=\sum_{i}z_{i}v$ for some $z_{i}\in U(\mathcal{L}_{-})$, where
$\deg_{W}z_{i}>1$. Then $W_{-p}L_{-n}v=(n-p)W_{-p-n}v+\sum z_{i}^{\prime}v$,
where $\deg_{W}z_{i}^{\prime}=\deg_{W}z_{i}>1$. We continue by induction to
prove that $W_{-p}L_{-n_{t}}\cdots L_{-n_{1}}v=\sum y_{i}v$, where $y_{i}v\in
B^{\prime}$ and $\deg_{W}y_{i}>1$. Acting with $W_{-m_{i}}$ for $i\neq j$ we
get $W_{-m_{k}}\cdots W_{-m_{1}}L_{-n_{t}}\cdots L_{-n_{1}}v=\sum W_{-m_{k}%
}\cdots W_{-m_{1}}y_{i}v $ what completes the proof.
\end{proof}

\begin{lemma}
\label{key'}Let $0\neq x\in L^{\prime}(c,h,h_{W})$ and $\deg_{W}\overline
{x}=k$. If $\overline{x}\notin\mathcal{W}$ and $n\in\mathbb{N}$ is the
smallest such that $L_{-n}$ occurs as a factor in one of the terms in
$\overline{x}$ then the part of $W_{n}x$ of the $W$-degree $k$ is given by
\[
n\left(  2h_{W}+\frac{n^{2}-1}{12}c\right)  \frac{\partial\overline{x}%
}{\partial L_{-n}}.
\]

\end{lemma}

\begin{proof}
We see from (\ref{w1}) that the part of $W$-degree $k$ comes from
$W_{n}\overline{x}$. Let $y=W_{-m_{k}}\cdots W_{-m_{1}}L_{-n_{t}}\cdots
L_{-n_{1}}v $, $m_{j}\neq p$ a component of $\overline{x}$. Then
\[
W_{n}y=\sum_{i=1}^{t}W_{-m_{k}}\cdots W_{-m_{1}}L_{-n_{t}}\cdots\lbrack
W_{n},L_{-n_{i}}]\cdots L_{-n_{1}}v
\]
and $n_{t}\geq n$. If $n_{i}=n$ we have
\[
\lbrack W_{n},L_{-n}]=n\left(  2W_{0}+\frac{n^{2}-1}{12}C\right)
\]
so this part contributes with $n(2h_{W}+\frac{n^{2}-1}{12}c)\frac
{\partial\overline{x}}{\partial L_{-n}}$.

If $n_{i}>n$ we get
\[
\lbrack W_{n},L_{-n_{i}}]=(n+n_{i})W_{n-n_{i}}.
\]
In case $n-n_{i}\neq p$ we get a component with $W$-degree $k+1$. Suppose
$n-n_{i}=p$. Then by Lemma \ref{sub}, this component can be substituted by the
sum of components with $W$-degree greater than $k$. This completes the proof.
\end{proof}

\begin{lemma}
\label{L-rstupanj}If $u$ is a subsingular vector such that $\overline
{u}=L_{-p}^{r}v$, then $\deg_{L}u=r$.
\end{lemma}

\begin{proof}
Consider $u$ as a singular vector in $L^{\prime}(c,h,h_{W})$. We use notation
$u=\sum_{n\geq0}\overline{u_{n}}$, where $\overline{u_{0}}=L_{-p}^{r}$, and
proceed by induction on $n$. Suppose $\deg_{L}\overline{u_{j}}\leq r$ for
$j<n$. Suppose to the contrary, that $\deg_{L}\overline{u_{n}}>r$. Denote by
$x$ the sum of all components in $u_{n}$ of $L$-degree at most $r$ and
$u_{n}=x+y$. Note that by definition $\overline{u_{n}}$ is homogeneous
respective to $W$-degree. Say $\deg_{W}\overline{u_{n}}=k$. Then $\deg
_{W}\overline{y}=k$. Let $m$ the smallest such that $L_{-m}$ occurs as a
factor in $y$. Obviously, $m<p$. We have $y=u-\sum_{i=0}^{n-1}\overline{u_{i}%
}-x$. By Lemma \ref{key'} the part of $W_{m}y$ of the $W$-degree $k$ is
$m(2h_{W}+\frac{m^{2}-1}{12}c)\frac{\partial\overline{y}}{\partial L_{-m}}%
\neq0$. Since $\deg_{L}y>r$, we have $\deg_{L}W_{m}y\geq r$. On the other hand
$\deg_{L}(\sum_{i=0}^{n-1}\overline{u_{i}}-x)\leq r$ so $\deg_{L}W_{m}%
(\sum_{i=0}^{n-1}\overline{u_{i}}-x)<r$. Since $W_{m}u=0$ in $L^{\prime
}(c,h,h_{W})$, we have $W_{m}y=-W_{m}(\sum_{i=0}^{n-1}\overline{u_{i}}-x)$
which is a contradiction.
\end{proof}

\begin{lemma}
\label{njnj}Let $0\neq x\in U(\mathcal{L})\left\{  u,u^{\prime}\right\}  $.
Then $W_{-p}$ or $L_{-p}^{k}$ for some $k\geq r$ occur as a factor in some
part of $\overline{x}$.
\end{lemma}

\begin{proof}
Suppose $x$ is homogeneous. Let $x=yu+zu^{\prime}$. If $y\in U(\mathcal{L}%
_{+})$ then $x\in U(\mathcal{L})u^{\prime}$, so we may apply Lemma \ref{nj}.
Furthermore, vectors from $\mathcal{L}_{0}u$ contain $L_{-p}^{r}v$ since
\[
W_{0}L_{-p}^{r}v=h_{W}L_{-p}^{r}+rpW_{-p}L_{-p}^{r-1}v,
\]
so $\overline{x}\in\mathbb{C}L_{-p}^{r}$. Finally, for $y\in U(\mathcal{L}%
_{-})$, $x$ must contain $L_{-p}^{k}$ for some $k\geq r$ as a factor in the
longest component of $\overline{x}$.
\end{proof}

\bigskip

For a PBW monomial $x=W_{-m_{s}}\cdots W_{-m_{1}}L_{-n_{t}}\cdots L_{-n_{1}}v
$ define $L_{-p}$-degree as a number of factors $L_{-n_{i}}=L_{-p}$, i.e.,
$\deg_{L_{-p}}x=\left\vert \left\{  i\in\left\{  1,\ldots,t\right\}
:n_{i}=p\right\}  \right\vert $.

\begin{lemma}
\label{baza2}Let $B$ the set of all PBW vectors $W_{-m_{s}}\cdots W_{-m_{1}%
}L_{-n_{t}}\cdots L_{-n_{1}}v$ with $L_{-p}$-degree strictly less than $r$ and
such that $m_{j}\neq p$. Then $B$ is a basis for a quotient module
$V(c,h,h_{W})/U(\mathcal{L})\left\{  u,u^{\prime}\right\}  $.
\end{lemma}

\begin{proof}
From Lemma \ref{njnj} follows linear independence of $B$. Next we show that
$a\in V(c,h,h_{W})/U(\mathcal{L})\left\{  u,u^{\prime}\right\}  $ can be
represented in $B$. Since $U(\mathcal{L})\left\{  u,u^{\prime}\right\}  $
contains $J^{\prime}(c,h,h_{W})$, we can represent $a$ in $B^{\prime}$, i.e.,
without $W_{-p}$ as a factor. Suppose $a$ contains monomial $W_{-m_{s}}\cdots
W_{-m_{1}}L_{-n_{t}}\cdots L_{-n_{1}}v$ with $L_{-p}$-degree greater than
$r-1$. We use relation
\begin{equation}
0=u=L_{-p}^{r}v+\sum x_{i},\quad\deg_{L}x_{i}\leq r,\deg_{L_{-p}}x_{i}<r
\label{8}%
\end{equation}
to eliminate every occurrence of a factor $L_{-p}^{r}$.

\textbf{Step 1} Acting with $W_{-m_{s}}\cdots W_{-m_{1}}\ $on (\ref{8}) we
immediately get
\[
W_{-m_{s}}\cdots W_{-m_{1}}L_{-p}^{r}v=\sum W_{-m_{s}}\cdots W_{-m_{1}}x_{i}%
\]
and $L_{-p}^{r}$ can not occur on the right side.

\textbf{Step 2} Let $n_{1}>p$. Acting with $L_{-n_{t}}\cdots L_{-n_{1}}$ on
(\ref{8}) we get
\[
L_{-n_{t}}\cdots L_{-n_{1}}L_{-p}^{r}v=\sum y_{i},\quad\deg_{L_{-p}}y_{i}%
=\deg_{L_{-p}}x_{i}%
\]
because $\left[  L_{-n},L_{-m}\right]  $ can not produce $L_{-p}$ if $n>p$.
Therefore $L_{-p}^{r}$ does not occur on the right side.

\textbf{Step 3} Let $n\leq p$. Acting with $L_{-n}$ on (\ref{8}) we get
\[
L_{-p}^{r}L_{-n}v=\sum z_{i},\quad\deg_{L}z_{i}\leq r+1,\deg_{L_{-p}}z_{i}\leq
r.
\]
the only part that can produce $L_{-p}^{r}$ comes from $L_{-n}wL_{-p}%
^{r-1}L_{-(p-n)}v$ for some $w\in\mathcal{W}$ (because $\deg_{L}u=r$). From
this we get $(p-2n)wL_{-p}^{r}v$ and we may apply Step 1 to eliminate this component.

Proceed by induction. Suppose that for every $j_{1},\ldots,j_{n-1}$, there
exist $u_{i}$, such that $\deg_{L}u_{i}\leq r+n-1$, $\deg_{L_{-p}}u_{i}<r$ and
$L_{-p}^{r}L_{-j_{n-1}}\cdots L_{-j_{1}}v=\sum u_{i}$. Acting with $L_{-j_{n}%
}$ we get
\[
L_{-p}^{r}L_{-j_{n}}\cdots L_{-j_{1}}v=\sum w_{i},\quad\deg_{L}w_{i}\leq
r+n,\deg_{L_{-p}}w_{i}\leq r.
\]
Part on the right side that can produce $L_{-p}^{r}$ has to come from monomial
$L_{-n}wL_{-p}^{r-1}L_{-k_{l}}\cdots L_{-k_{1}}v$ with $l\leq n$ for some
$w\in U(\mathcal{L}_{-})$. From this we get $w^{\prime}L_{-p}^{r}%
L_{-k_{l-1}^{\prime}}\cdots L_{-k_{1}^{\prime}}v$ and by induction this
component can be replaced with a vector from $B$. This completes the proof.
\end{proof}

\begin{theorem}
Let $2h_{W}+\frac{p^{2}-1}{12}c=0$. If $V(c,h,h_{W})$ contains a subsingular
vector $u$ such that $\overline{u}=L_{-p}^{r}v$ for some $r\in\mathbb{N}$,
then $J(c,h,h_{W})=U(\mathcal{L})\left\{  u,u^{\prime}\right\}  $ is the
maximal submodule and module $L(c,h,h_{W})=V(c,h,h_{W})/J(c,h,h_{W})$ is
irreducible with PBW basis
\[
B=\left\{  x=W_{-m_{s}}\cdots W_{-m_{1}}L_{-n_{t}}\cdots L_{-n_{1}}v:m_{j}\neq
p,\deg_{L_{-p}}x<r\right\}
\]
and a character
\begin{gather*}
\operatorname*{char}L(c,h,h_{W})=\\
=q^{h}(1-q^{p})(1-q^{rp})\sum_{n\geq0}P_{2}(n)q^{n}=q^{h}(1-q^{p}%
)(1-q^{rp})\prod\limits_{k\geq1}(1-q^{k})^{-2}.
\end{gather*}

\end{theorem}

\begin{proof}
Let $0\neq x\in V(c,h,h_{W})$ such that $U(\mathcal{L}_{+})x\subseteq
J(c,h,h_{W})$. We may view $x$ as a homogeneous vector presented in $B$.
Analogous to proof of Theorem \ref{w irr}, one can show that $\overline
{x}=\mathbb{C}L_{-p}^{s}v$ for some $s\in\mathbb{N}$. From Lemma \ref{baza2}
we conclude that $s<r$ so we have found $x\in V(c,h,h_{W})$ such that
$U(\mathcal{L}_{+})x\in U(\mathcal{L})u^{\prime}$ and $\overline{x}=L_{-p}%
^{s}v$ for $s<r$. This is a contradiction to Theorem \ref{dov}, proving
irreducibility of $L(v,h,h_{W})$.

Simple combinatorics on $B$ shows that
\begin{gather*}
\operatorname*{char}L(c,h,h_{W})=\\
=q^{h}\left(  \sum_{n\geq0}P_{2}(n)q^{n}-q^{p}\sum_{n\geq0}P_{2}%
(n)q^{n}-q^{rp}\sum_{n\geq0}P_{2}(n)q^{n}+q^{(r+1)p}\sum_{n\geq0}P_{2}%
(n)q^{n}\right) \\
=q^{h}(1-q^{p})(1-q^{rp})\sum_{n\geq0}P_{2}(n)q^{n}=q^{h}(1-q^{p}%
)(1-q^{rp})\prod\limits_{k\geq1}(1-q^{k})^{-2}%
\end{gather*}
so all claims are proved.
\end{proof}

This completes the proof of Theorem \ref{glavni}.

Note that $\operatorname*{char}J(c,h,h_{W})=\operatorname*{char}%
V(c,h,h_{W})-\operatorname*{char}L(c,h,h_{W})$ so
\begin{align*}
\operatorname*{char}J(c,h,h_{W})  &  =q^{h+p}(1+q^{(r-1)p}-q^{rp})\sum
_{n\geq0}P_{2}(n)q^{n}=\\
&  =q^{h+p}(1+q^{(r-1)p}-q^{rp})\prod\limits_{k\geq1}(1-q^{k})^{-2}.
\end{align*}
From there we get
\begin{align*}
\operatorname*{char}J(c,h,h_{W})/J^{\prime}(c,h,h_{W})  &  =q^{h+rp}%
(1-q^{p})\sum_{n\geq0}P_{2}(n)q^{n}=\\
&  =q^{h+rp}(1-q^{p})\prod\limits_{k\geq1}(1-q^{k})^{-2}%
\end{align*}
so
\[
J(c,h,h_{W})/J^{\prime}(c,h,h_{W})\cong L^{\prime}(c,h+rp,h_{W}%
)=L(c,h+rp,h_{W}).
\]
The equation above is due to
\[
h+rp=h_{W}+\frac{(13p+1)(p-1)}{12}+\frac{(1+r)p}{2},r\in\mathbb{N}%
\]
and Corollary \ref{wkrit}.

Since $u\operatorname{mod}J^{\prime}(c,h,h_{W})$ generates $J(c,h,h_{W}%
)/J^{\prime}(c,h,h_{W})$ we get $W_{0}u\in h_{W}u+J^{\prime}(c,h,h_{W})$.

Likewise,
\[
\operatorname*{char}L^{\prime}(c,h,h_{W})/L(c,h+rp,h_{W})=\operatorname*{char}%
L(c,h,h_{W}).
\]

\begin{corollary}
\label{L'}If Verma module $V(c,h,h_{W})$ contains a subsingular vector $u$
such that $\overline{u}=L_{-p}^{r}v$ then the following short sequence is
exact%
\[
0\rightarrow L(c,h+rp,h_{W})\rightarrow L^{\prime}(c,h,h_{W})\rightarrow
L(c,h,h_{W})\rightarrow0.
\]
Since the other two modules are irreducible, $L(c,h+rp,h_{W})$ is the only
nontrivial submodule in $L^{\prime}(c,h,h_{W})$.
\end{corollary}

\bigskip

Next we present examples of singular and subsingular vectors in simplest cases.

Singular vector $u^{\prime}\in\mathcal{W}:\medskip$

$%
\begin{tabular}
[c]{|c|c|}\hline
$h_{W}=0$ & $W_{-1}v$\\\hline
$c=-8h_{W}$ & $\left(  W_{-2}-\frac{3}{4h_{W}}W_{-1}^{2}\right)  v$\\\hline
$c=-3h_{W}$ & $\left(  W_{-3}-\frac{2}{h_{W}}W_{-2}W_{-1}+\frac{1}{h_{W}^{2}%
}W_{-1}^{3}\right)  v$\\\hline
$c=-\frac{8}{5}h_{W}$ & $%
\begin{array}
[c]{c}%
\left(  W_{-4}-\frac{5}{2h_{W}}W_{-3}W_{-1}-\frac{15}{16h_{W}}W_{-2}%
^{2}+\right. \\
\left.  +\frac{125}{32h_{W}^{2}}W_{-2}W_{-1}^{2}-\frac{375}{256h_{W}^{3}%
}W_{-1}^{4}\right)  v
\end{array}
$\\\hline
$c=-h_{W}$ & $%
\begin{array}
[c]{c}%
(W_{-5}-\frac{3}{h_{W}}W_{-4}W_{-1}-\frac{2}{h_{W}}W_{-3}W_{-2}+\frac
{21}{4h_{W}^{2}}W_{-3}W_{-1}^{2}+\\
+\frac{3}{h_{W}^{2}}W_{-2}^{2}W_{-1}+\frac{13}{2h_{W}^{3}}W_{-2}W_{-1}%
^{3}-\frac{39}{20h_{W}^{3}}W_{-1}^{5})v
\end{array}
$\\\hline
\end{tabular}
$

\bigskip Let $r=1$. Then a subsingular vector $u$ is actually a singular
vector on level $p$. By Lemma \ref{L-rstupanj} we can write $u=w_{0}%
v+\sum_{i=1}^{p-1}w_{i}L_{-i}v+L_{-p}v$ where $w_{i}\in U\left(
\mathcal{L}_{-}\right)  _{p-i}$, $\deg_{L}w_{i}=0$. Acting with $W_{p-n}$ we
get recursive relation for $w_{n}$, $n=1,\ldots,p-1$. We use identity
$[W_{n},L_{-n}]v=2h_{W}n\frac{n^{2}-p^{2}}{1-p^{2}}v$ since $2h_{W}%
+\frac{p^{2}-1}{12}c=0$.
\begin{align*}
0  &  =W_{p-1}u=w_{p-1}[W_{p-1},L_{-\left(  p-1\right)  }]v+[W_{p-1}%
,L_{-p}]v=\\
&  =\frac{2p-1}{p+1}2h_{W}w_{p-1}v+\left(  2p-1\right)  W_{-1}v
\end{align*}
so $w_{p-1}=-\frac{p+1}{2h_{W}}W_{-1}$. Next we act with $W_{p-2}$ to get
$w_{p-2}$ and so on. In general:
\[
w_{n}=\frac{p^{2}-1}{2h_{W}n(n^{2}-p^{2})}\left(  \sum_{i=n+1}^{p-1}\left(
n+i\right)  w_{i}W_{-\left(  i-n\right)  }+\left(  n+p\right)  W_{-\left(
p-n\right)  }\right)  .
\]

We assume $W_{-p}v$ does not occur in $u$ (see Remark \ref{uni}).

\newpage

(Sub)singular vector $u$ when $r=1:\medskip$%

\begin{tabular}
[c]{|c|c|}\hline
$h=h_{W}=0$ & $L_{-1}v$\\\hline
$h=h_{W}+\frac{9}{4}$ & $\left(  L_{-2}-\frac{3}{2h_{W}}W_{-1}L_{-1}%
+\frac{12h_{W}+39}{16h_{W}^{2}}W_{-1}^{2}\right)  v$\\\hline
$h=h_{W}+\frac{20}{3}$ & $%
\begin{array}
[c]{c}%
\left(  L_{-3}-\frac{2}{h_{W}}W_{-1}L_{-2}-\frac{2}{h_{W}}W_{-2}L_{-1}%
+\frac{3}{h_{W}^{2}}W_{-1}^{2}L_{-1}\right. \\
\left.  +\frac{58}{3h_{W}^{2}}W_{-2}W_{-1}-\frac{2h_{w}+52}{3h_{W}^{3}}%
W_{-1}^{3}\right)  v
\end{array}
$\\\hline
$h=h_{W}+\frac{53}{4}$ & $%
\begin{array}
[c]{c}%
\left(  L_{-4}-\frac{5}{2h_{W}}W_{-1}L_{-3}-\frac{15}{8h_{W}}W_{-2}%
L_{-2}+\frac{125}{32h_{W}^{2}}W_{-1}^{2}L_{-2}\right. \\
-\frac{5}{2h_{W}}W_{-3}L_{-1}+\frac{125}{16h_{W}^{2}}W_{-2}W_{-1}L_{-1}%
-\frac{375}{64h_{W}^{3}}W_{-1}^{3}L_{-1}\\
+\frac{325+20h_{W}}{8h_{W}^{2}}W_{-3}W_{-1}+\frac{8125}{64h_{W}^{3}}%
W_{-2}W_{-1}^{2}\\
\left.  +\frac{975+60h_{W}}{64h_{W}^{2}}W_{-2}^{2}+\frac{1125}{256h_{W}^{3}%
}\left(  \frac{65}{4h_{W}}+1\right)  W_{-1}^{4}\right)  v
\end{array}
$\\\hline
\end{tabular}

$\medskip$

Using computer, we have found formulas for singular vectors $u$ up to level 8.

\bigskip

Subsingular vectors $u$ in $V(c,\frac{1-r}{2},0):\medskip$%

\begin{tabular}
[c]{|c|c|}\hline
$h=h_{W}=0$ & $L_{-1}v$\\\hline
$h=-\frac{1}{2}$ & $\left(  L_{-1}^{2}+\frac{6}{c}W_{-2}\right)  v$\\\hline
$h=-1$ & $\left(  L_{-1}^{3}+\frac{12}{c}W_{-3}+\frac{24}{c}W_{-2}%
L_{-1}\right)  v$\\\hline
$h=-\frac{3}{2}$ & $\left(  L_{-1}^{4}+\frac{36}{c}W_{-4}+\frac{60}{c}%
W_{-3}L_{-1}+\frac{108}{c^{2}}W_{-2}^{2}+\frac{60}{c}W_{-2}L_{-1}^{2}\right)
v$\\\hline
$h=-2$ & $%
\begin{array}
[c]{c}%
\left(  L_{-1}^{5}+\frac{144}{c}W_{-5}+\frac{48}{c}W_{-4}L_{-1}+\frac
{2304}{c^{2}}W_{-3}W_{-2}+\right. \\
\left.  +\frac{180}{c}W_{-3}L_{-1}^{2}+\frac{3312}{c^{2}}W_{-2}^{2}%
L_{-1}+\frac{120}{c}W_{-2}L_{-1}^{3}\right)  v
\end{array}
$\\\hline
\end{tabular}

\begin{conjecture}
\label{slutnja}Let $u\in V(c,\frac{1-r}{2},0)$ be a subsingular vector. Then
$L_{-k}$ for $k>1$ does not occur as a factor in $u$. Specially,
$W_{0}u,W_{-1}u\in J^{\prime}(c,\frac{1-r}{2},0)$.
\end{conjecture}

\bigskip

Subsingular vector $u$ in $V(-8h_{W},h_{W}+\frac{5}{4},h_{W})$ ($p=r=2$ case):%
\begin{gather*}
\left(  L_{-2}^{2}-\frac{3}{4h_{W}}W_{-4}-\left(  \frac{3}{2h_{W}^{2}}%
+\frac{3}{2h_{W}}\right)  W_{-3}W_{-1}+\frac{3}{2h_{W}}W_{-3}L_{-1}-\frac
{3}{2h_{W}}W_{-1}L_{-3}\right. \\
-\frac{3}{h_{W}}W_{-1}L_{-2}L_{-1}+\left(  \frac{3}{2h_{W}}+\frac{39}%
{4h_{W}^{2}}\right)  W_{-1}^{2}L_{-2}+\frac{9}{4h_{W}^{2}}W_{-1}^{2}L_{-1}%
^{2}+\\
\left.  -\left(  \frac{9}{4h_{W}^{2}}+\frac{117}{8h_{W}^{3}}\right)
W_{-1}^{3}L_{-1}+\left(  \frac{135}{32h_{W}^{4}}+\frac{153}{32h_{W}^{3}}%
+\frac{9}{8h_{W}^{2}}\right)  W_{-1}^{4}\right)  v
\end{gather*}

\bigskip

Based on these examples, we state

\begin{conjecture}
\label{slut}Suppose $2h_{W}+\frac{p^{2}-1}{12}c=0$ for some $p\in\mathbb{N}$.
Then $L^{\prime}(c,h,h_{W})$ is reducible if and only if $h=h_{W}%
+\frac{(13p+1)(p-1)}{12}+\frac{(1-r)p}{2}$.
\end{conjecture}

\begin{remark}
Let $2h_{W}+\frac{p^{2}-1}{12}c=0$. Then $U(\mathcal{L})u^{\prime
}=V(c,h+p,h_{W})\subseteq V(c,h,h_{W})$. However, $V(c,h+p,h_{W})$ also
contains a singular vector $(u^{\prime})^{2}$ such that $U(\mathcal{L}%
)(u^{\prime})^{2}=V(c,h+2p,h_{W})$ and so on. Therefore,
\[
\cdots\supseteq V(c,h,h_{W})\supseteq V(c,h+p,h_{W})\supseteq\cdots\supseteq
V(c,h+kp,h_{W})\supseteq\cdots
\]
If $h\neq h_{W}+\frac{(13p+1)(p-1)}{12}+\frac{(1-r)p}{2}$ for all
$n\in\mathbb{Z}$, then all of the subquotients $V(c,h+kp,h_{W}%
)/V(c,h+(k+1)p,h_{W})$ are irreducible modules $L(c,h+kp,h_{W}) $.
\end{remark}

\section{Vertex operator algebra associated to $W(2,2)$ and intertwining
operators\label{VOA}}

Vertex operator algebras (VOAs) are a fundamental class of algebraic
structures which have arisen in mathematics and physics a few decades ago. The
notion of VOA arose from the problem of realizing the monster sporadic group
as a symmetry group of a certain infinite-dimensional vector space (see
\cite{Frenkel - Lepowsky - Meurman}). An interested reader should consult
\cite{Frenkel - Huang - Lepowsky} or \cite{Lepowsky-Li} for a detailed
approach. Here we present only basic definitions which should suffice our needs.

For any algebraic expression $z$ we set $\delta\left(  z\right)  =\sum
_{n\in\mathbb{Z}}z^{n}$ provided that this sum makes sense. This is the formal
analogue of the $\delta$-distribution at $z=1$; in particular, $\delta\left(
z\right)  f\left(  z\right)  =\delta\left(  z\right)  f\left(  1\right)  $ for
any $f$ for which these expressions are defined.

\begin{definition}
A \textbf{vertex operator algebra} $\left(  V,Y,\mathbf{1}\right)  $ is a
$\mathbb{Z}$-graded vector space (graded by \textit{weights}) $V=\bigoplus
\limits_{n\in\mathbb{Z}}V_{\left(  n\right)  }$ such that $\dim V_{\left(
n\right)  }<\infty$ for $n\in\mathbb{Z}$ and $V_{\left(  n\right)  }=0$ for
$n$ sufficiently small, equipped with a linear map $V\otimes V\rightarrow
V\left[  \left[  z,z^{-1}\right]  \right]  $, or equivalently,
\begin{align*}
V  &  \rightarrow\left(  \operatorname*{End}V\right)  \left[  \left[
z,z^{-1}\right]  \right] \\
v  &  \mapsto Y\left(  v,z\right)  =\sum_{n\in\mathbb{Z}}v_{n}z^{-n-1}\text{
(where }v_{n}\in\operatorname*{End}V\text{),}%
\end{align*}
$Y\left(  v,z\right)  $ denoting the \textit{vertex operator associated} with
$v$, and equipped with two distinguished homogeneous vectors $\mathbf{1}$ (the
\textit{vacuum}) and $\omega\in V$. The following conditions are assumed for
$u,v\in V:$%
\[
u_{n}v=0\text{ for }n\text{ sufficiently large;}%
\]%
\[
Y\left(  \mathbf{1},z\right)  =1\left(  =\operatorname*{id}\nolimits_{V}%
\right)  ;
\]
the \textit{creation property} holds:
\[
Y\left(  v,z\right)  \mathbf{1}\in V\left[  \left[  z\right]  \right]  \text{
and }\lim_{z\rightarrow0}Y\left(  v,z\right)  \mathbf{1}=0
\]
(that is, $Y\left(  v,z\right)  \mathbf{1}$ involves only nonnegative integral
powers of $z$ and the constant term is $v$); the \textit{Jacobi identity}:
\begin{align}
&  z_{0}^{-1}\delta\left(  \frac{z_{1}-z_{2}}{z_{0}}\right)  Y(u,z_{1}%
)Y(v,z_{2})-z_{0}^{-1}\delta\left(  \frac{z_{2}-z_{1}}{-z_{0}}\right)
Y(v,z_{2})Y(u,z_{1})\label{jacobi}\\
&  =z_{2}^{-1}\delta\left(  \frac{z_{1}-z_{0}}{z_{2}}\right)  Y(Y(u,z_{0}%
)v,z_{2});\nonumber
\end{align}
the Virasoro algebra relations:
\[
\left[  L_{m},L_{n}\right]  =\left(  m-n\right)  L_{m+n}+\frac{m^{3}-m}%
{12}\delta_{m+n,0}\left(  \text{rank }V\right)
\]
for $m,n\in\mathbb{Z}$, where
\[
L_{n}=\omega_{n+1}\text{ for }n\in\mathbb{Z}\text{, i.e., }Y\left(
\omega,z\right)  =\sum_{n\in\mathbb{Z}}L_{n}z^{-n-2}%
\]
and rank~$V\in\mathbb{C}$, $L_{0}v=nv=\left(  wt\ v\right)  v$ for
$n\in\mathbb{Z}$ and $v\in V_{\left(  n\right)  };$
\[
\frac{d}{dz}Y\left(  v,z\right)  =Y\left(  L_{-1}v,z\right)
\]
(the $L_{-1}$-derivative property).
\end{definition}

\begin{definition}
Given a VOA $\left(  V,Y,\mathbf{1}\right)  ,$ a \textbf{module} $\left(
W,\mathcal{Y}\right)  $ for $V$ is a $\mathbb{Q}$-graded vector space
$W=\bigoplus\limits_{n\in\mathbb{Q}}W_{\left(  n\right)  }$ such that $\dim
W_{\left(  n\right)  }<\infty$ for $n\in\mathbb{Q}$, and $W_{\left(  n\right)
}=0$ for $n$ sufficiently small, equipped with a linear map $V\otimes
W\rightarrow W\left[  \left[  z,z^{-1}\right]  \right]  $, or equivalently,
\begin{align*}
V  &  \rightarrow\left(  \operatorname*{End}W\right)  \left[  \left[
z,z^{-1}\right]  \right] \\
v  &  \mapsto\mathcal{Y}\left(  v,z\right)  =\sum_{n\in\mathbb{Z}}%
v_{n}z^{-n-1}\text{ (where }v_{n}\in\operatorname*{End}W\text{),}%
\end{align*}
$\mathcal{Y}\left(  v,z\right)  $ denoting the \textit{vertex operator
associated} with $v$. The Virasoro algebra relations hold on $W$ with scalar
equal to rank $V$:
\[
\left[  L_{m},L_{n}\right]  =\left(  m-n\right)  L_{m+n}+\frac{m^{3}-m}%
{12}\delta_{m+n,0}\left(  \text{rank }V\right)
\]
for $m,n\in\mathbb{Z}$, where
\[
L_{n}=\omega_{n+1}\text{ for }n\in\mathbb{Z}\text{, i.e., }\mathcal{Y}\left(
\omega,z\right)  =\sum_{n\in\mathbb{Z}}L_{n}z^{-n-2}%
\]
$L_{0}w=nw$ for $n\in\mathbb{Q}$ and $w\in W_{\left(  n\right)  }.$ For
$u,v\in V$ and $w\in W$ the following properties hold:

\begin{description}
\item[(i)] \textit{Truncation property}: $v_{n}w=0$ for $n$ sufficiently large
and $Y\left(  \mathbf{1},z\right)  =1$;

\item[(ii)] $L_{-1}$-derivative property
\[
\frac{d}{dz}\mathcal{Y}\left(  v,z\right)  =\mathcal{Y}\left(  L_{-1}%
v,z\right)
\]

\item[(iii)] The Jacobi identity
\begin{align*}
&  z_{0}^{-1}\delta\left(  \frac{z_{1}-z_{2}}{z_{0}}\right)  \mathcal{Y}%
(u,z_{1})\mathcal{Y}(v,z_{2})-z_{0}^{-1}\delta\left(  \frac{z_{2}-z_{1}%
}{-z_{0}}\right)  \mathcal{Y}(v,z_{2})\mathcal{Y}(u,z_{1})\\
&  =z_{2}^{-1}\delta\left(  \frac{z_{1}-z_{0}}{z_{2}}\right)  \mathcal{Y}%
(Y(u,z_{0})v,z_{2});
\end{align*}
($Y\left(  u,z_{0}\right)  $ is the operator associated with $V$).
\end{description}
\end{definition}

A VOA\ $V$ is said to be \textbf{rational} if $V$ has only finitely many
irreducible modules and every finitely generated module is a direct sum of irreducibles.

\begin{definition}
Let $V=(V,Y,\mathbf{1})$ be a vertex operator algebra, and $\left(
W_{i},Y_{i}\right)  $, $i=1,2,3$ three (not necessarily distinct) $V$-modules.
\textbf{Intertwining operator} of type $\binom{W_{3}}{W_{1}\text{\quad}W_{2}}$
is a linear map $W_{1}\otimes W_{2}\rightarrow W_{3}\{z\}=\left\{  \sum
_{n\in\mathbb{Q}}u_{n}z^{n}:u_{n}\in W_{k}\right\}  $, or equivalently,
\begin{align*}
W_{1}  &  \rightarrow\left(  \operatorname*{Hom}\left(  W_{2},W_{3}\right)
\right)  \left\{  z\right\} \\
w  &  \mapsto\mathcal{I}\left(  w,z\right)  =\sum_{n\in\mathbb{Q}}%
w_{n}z^{-n-1}\text{, with }w_{n}\in\operatorname*{Hom}\left(  W_{2}%
,W_{3}\right)  .
\end{align*}
For any $v\in V$, $u\in W_{1}$, $w\in W_{2}$ the following conditions are satisfied:

\begin{description}
\item[(i)] \textit{Truncation property} - $u_{n}v=0$ for $n$ sufficiently large;

\item[(ii)] $L_{-1}$\textit{-derivative property} - $\mathcal{I}%
(L_{-1}u,z)=\frac{d}{dz}\mathcal{I}(u,z)$;

\item[(iii)] The \textit{Jacobi identity}
\begin{align*}
&  z_{0}^{-1}\delta\left(  \frac{z_{1}-z_{2}}{z_{0}}\right)  Y_{3}%
(v,z_{1})\mathcal{I}(u,z_{2})w-z_{0}^{-1}\delta\left(  \frac{z_{2}-z_{1}%
}{-z_{0}}\right)  \mathcal{I}(u,z_{2})Y_{2}(v,z_{1})w\\
&  =z_{2}^{-1}\delta\left(  \frac{z_{1}-z_{0}}{z_{2}}\right)  \mathcal{I}%
(Y_{1}(v,z_{0})u,z_{2})w.
\end{align*}

\end{description}
\end{definition}

\bigskip

Let $c\neq0$. It was shown in Zhang-Dong \cite{Zhang-Dong} that $L(c,0,0)$ is
the only quotient of $V(c,0,0)$ with the structure of a vertex operator
algebra (VOA). This algebra is always irrational and all of its irreducible
representations are known:

\begin{theorem}
[\cite{Zhang-Dong}]Let $c\neq0$. Then

\begin{enumerate}
\item There is a unique VOA structure on $L(c,0,0)$ with the vacuum vector $v
$, and the Virasoro element $\omega=L_{-2}v$. $L(c,0,0)$ is generated with
$\omega$ and $x=W_{-2}v$ and $Y\left(  \omega,z\right)  =\sum_{n\in\mathbb{Z}%
}L_{n}z^{-n-2}$, $Y(W_{-2}v,z)=\sum_{n\in\mathbb{Z}}W_{n}z^{-n-2}$.

\item Any quotient module of $V(c,h,h_{W})$ is an $L(c,0,0)$-module, and
$\{L(c,h,h_{W}):h,h_{W}\in\mathbb{C\}}$ gives a complete list of irreducible
$L(c,0,0)$-modules up to isomorphism.
\end{enumerate}
\end{theorem}

Now we present a realization of intermediate series via intertwining operators
for $L(c,0,0)$-modules.

Let $M(c,h,h_{W})$ denote a highest weight module. Suppose a nontrivial
intertwining operator $\mathcal{I}$ of type $\binom{M(c,h_{3}h_{W}^{\prime}%
)}{L^{\prime}(c,h_{1},0)\text{\quad}M(c,h_{2},h_{W})}$ exists. Let $h_{1}\neq0
$ and $v\in L^{\prime}(c,h_{1},0)$ the highest weight vector. Let
$\mathcal{I}(v,z)=z^{-\alpha}\sum_{n\in\mathbb{Z}}v_{(n)}z^{-n-1}$ for
$\alpha=h_{1}+h_{2}-h_{3}$. Recall that $W_{0}v=W_{-1}v=0$. Then
\begin{align*}
\left[  L_{m},v_{(n)}\right]   &  =\sum_{i\geq0}\binom{m+1}{i}\left(
L_{i-1}v\right)  _{\left(  m+n-i+1\right)  }=\\
&  =\left(  L_{-1}v\right)  _{\left(  m+n+1\right)  }+\left(  m+1\right)
\left(  L_{0}v\right)  _{\left(  m+n\right)  }=\\
&  =-\left(  \alpha+n+m+1\right)  v_{\left(  m+n\right)  }+\left(  m+1\right)
h_{1}v_{\left(  m+n\right)  }=\\
&  =-\left(  n+\alpha+\left(  1+m\right)  \left(  1-h_{1}\right)  \right)
v_{\left(  m+n\right)  }%
\end{align*}
and
\begin{align*}
\left[  W_{m},v_{(n)}\right]   &  =\sum_{i\geq0}\binom{m+1}{i}\left(
W_{i-1}v\right)  _{\left(  m+n-i+1\right)  }=\\
&  =\left(  W_{-1}v\right)  _{\left(  m+n+1\right)  }+\left(  m+1\right)
\left(  W_{0}v\right)  _{\left(  m+n\right)  }=0
\end{align*}
so components $v_{(n)}$ span a module from intermediate series $V_{\alpha
,\beta,0}^{\prime}$ for $\beta=1-h_{1}$. Hence we have a nontrivial
$\mathcal{L}$-operator
\[
\Phi:V_{\alpha,\beta,0}^{\prime}\otimes M(c,h_{2},h_{W})\rightarrow
M(c,h_{3},h_{W}^{\prime}),\quad\Phi(v_{(n)}\otimes x)=v_{(n)}x.
\]
Since dimensions of weight subspaces are infinite in $V_{\alpha,\beta
,0}^{\prime}\otimes M(c,h_{2},h_{W})$ and finite in $M(c,h_{3},h_{W}^{\prime
})$, we conclude that $V_{\alpha,\beta,0}^{\prime}\otimes M(c,h_{2},h_{W})$ is reducible.

Let us mark some intertwining operators. Since $M(c,h,h_{W})$ is
$L(c,0,0)$-module, there exists intertwining operator of type $\binom
{M(c,h,h_{W})}{L(c,0,0)\text{\quad}M(c,h,h_{W})}$ and the transposed operator
of type $\binom{M(c,h,h_{W})}{M(c,h,h_{W})\text{\quad}L(c,0,0)}$. In
particular, operators of type
\[
\binom{L(c,h,0)}{L(c,h,0)\text{\quad}L(c,0,0)}\text{ and }\binom{L^{\prime
}(c,h,0)}{L^{\prime}(c,h,0)\text{\quad}L(c,0,0)}%
\]
exist for all $h$.

\section{Irreducibility of a module $V_{\alpha,\beta,0}^{\prime}\otimes
L(c,h,h_{W})$\label{irr tezor}}

We state the main result of this section immediately:

\begin{theorem}
\label{main}Module $V_{\alpha,\beta,0}^{\prime}\otimes L(c,h,h_{W})$ is
irreducible if and only if there is a subsingular vector $u$ in $V(c,h,h_{W})
$ such that $\overline{u}=L_{-p}^{r}v$, and if $\alpha+(1-p)\beta
\notin\mathbb{Z}$.
\end{theorem}

We shall prove this theorem in several steps, using analogous approach as in
the Virasoro case (see \cite{Radobolja2}). First we show that some relation in
$L(c,h,h_{W})$ is needed to obtain irreducibility.

\begin{theorem}
\label{ver}Module $V_{\alpha,\beta,0}^{\prime}\otimes V(c,h,h_{W})$ is reducible.
\end{theorem}

\begin{proof}
Let $v$ the highest weight vector in $V(c,h,h_{W})$. We show $v_{k}\otimes v$
generates a proper submodule for any $k\in\mathbb{Z}$. Assume otherwise. Then
$w$ exists in $U(\mathcal{L})$,$\mathcal{\ }$such that $w(v_{k}\otimes
v)=v_{k-1}\otimes v$. Since $U(\mathcal{L}_{+})v=0$, we have $v_{k-1}\otimes
v=\sum_{i=0}^{n}w_{i+1}(v_{k+i}\otimes v)$ for some $0\neq w_{j}\in
U(\mathcal{L}_{-})_{-j}$. But then $v_{k+n}\otimes w_{n+1}v=0$ which means
that $w_{n+1}v=0$. This is impossible since Verma module is free over
$U(\mathcal{L}_{-})$.
\end{proof}

Next we prove irreducibility criterion analogous to the one proven in
\cite{Zhang} and generalized in \cite{Radobolja2}.

\begin{theorem}
\label{m}Module $V_{\alpha,\beta,0}^{\prime}\otimes L(c,h,h_{W})$ is
irreducible if and only if it is cyclic on every vector $v_{k}\otimes v$.
\end{theorem}

\begin{proof}
We follow the proof of Theorem 3.2 in \cite{Radobolja2}. The only if part is
trivial. Assume $V_{\alpha,\beta,0}^{\prime}\otimes L(c,h,h_{W})$ is cyclic on
every $v_{k}\otimes v$. Let $U$ be a submodule and $0\neq x\in U$ homogeneous
vector. Then
\[
x=v_{m-n}\otimes x_{0}+\cdots+v_{m}\otimes x_{n}%
\]
for some $x_{j}\in V(c,h,h_{W})_{j}$. We use induction on $n$ to show there is
$v_{k}\otimes v\in U$ for some $k\in\mathbb{Z}$. If $n=0$, $x_{n}\in
\mathbb{C}v$ and we are done. Let $n>0$. Our strategy is to act on $x$ by
$U(\mathcal{L})$ in order to 'shorten it' and apply induction. Recall
$W_{k}v_{i}=0$ for any $k,i$. If $W_{k}x_{n}\neq0$ for $k\in\mathbb{N}$, we
have
\[
W_{k}x=v_{m-n+k}\otimes y_{0}+\cdots+v_{m}\otimes y_{n-k}\neq0
\]
where $y_{j}=W_{k}x_{j+k}\in V(c,h,h_{W})_{j}$. By inductive hypothesis, now
there must be some $v_{k}\otimes v\in U$. Assume, therefore, that $W_{k}%
x_{n}=0$ for all $k\in\mathbb{N}$. Since $W_{1}$, $W_{2}$, $L_{1}$ and $L_{2}$
generate $U(\mathcal{L}_{+})$, vectors $L_{1}x_{n}$ and $L_{2}x_{n}$ can not
both equal zero, for otherwise $x_{n}$ would be a singular vector in
$L(c,h,h_{W})$ other than $v$. But now we can follow the proof of Theorem 3.2
in \cite{Radobolja2}. This completes the proof.
\end{proof}

Define $U_{n}:=U(\mathcal{L})(v_{n}\otimes v)$. Proof of Theorem \ref{m}
actually shows

\begin{corollary}
\label{t}Let $M$ a nontrivial submodule in $V_{\alpha,\beta,0}^{\prime}\otimes
L(c,h,h_{W})$. Then $M$ contains $U_{n}$ for some $n\in\mathbb{Z}$.
\end{corollary}

Module $V_{\alpha,\beta,0}^{\prime}\otimes L(c,h,h_{W})$ is irreducible if and
only if $U_{n}=U_{n+1}$ for all $n\in\mathbb{Z}$. Since $L_{1}(v_{n}\otimes
v)=-(n+\alpha+2\beta)v_{n+1}\otimes v$, we get $U_{n}\supseteq U_{n+1}$ for
$n\neq-\alpha-2\beta$. To obtain the other inclusion, we need a subsingular vector.

\begin{lemma}
\label{teh}Let $k=k_{1}+\cdots k_{s}$ for $k_{i}\in\mathbb{N}$. Then
\[
L_{-k_{s}}\cdots L_{-k_{1}}(v_{n}\otimes v)=\sum_{i=0}^{k}u_{i}(v_{n-i}\otimes
v)
\]
for some $u_{i}\in U(\operatorname*{Vir}_{-})$.
\end{lemma}

\begin{proof}
From (\ref{1})\textbf{\ }we get
\[
L_{-k_{s}}\cdots L_{-k_{1}}(v_{n}\otimes v)=\sum_{i=0}^{k}v_{n-i}\otimes
L_{-j_{m}}\cdots L_{-j_{1}}v
\]
for $0\leq j_{1}+\cdots+j_{m}=k-i$. Next we show
\begin{equation}
v_{n-i}\otimes L_{-j_{m}}\cdots L_{-j_{1}}v=\sum_{j=0}^{k-i}x_{j}%
(v_{n-k+j}\otimes v) \label{hh}%
\end{equation}
for some $x_{j}\in U(\mathcal{L}_{-})_{-j}$. But since
\begin{gather*}
v_{n-i}\otimes L_{-j_{m}}\cdots L_{-j_{1}}v=L_{-j_{m}}(v_{n-i}\otimes
L_{-j_{m-1}}\cdots L_{-j_{1}}v)+\\
+(n-i+\alpha+\beta-j_{m}\beta)v_{n-i-j_{m}}\otimes L_{-j_{m-1}}\cdots
L_{-j_{1}}v
\end{gather*}
we may use induction on $m$ to complete the proof.
\end{proof}

\begin{theorem}
\label{tenz irr}Let $u\in V(c,h,h_{W})$ a subsingular vector such that
$\overline{u}=L_{-p}^{r}$. If $\alpha+(1-p)\beta\notin\mathbb{Z}$, module
$V_{\alpha,\beta,0}^{\prime}\otimes L(c,h,h_{W})$ is irreducible.
\end{theorem}

\begin{proof}
For any $n\in\mathbb{N}$ we need to find $x\in U(\mathcal{L})$ such that
$x(v_{n}\otimes v)=v_{n-1}\otimes v$. Since $uv=L_{-p}^{r}v+\sum
w_{i}L_{-k_{s}}\cdots L_{-k_{1}}v=0$, we have
\[
u(v_{n+rp-1}\otimes v)=L_{-p}^{r}(v_{n+rp-1}\otimes v)+\sum w_{i}L_{-k_{s}%
}\cdots L_{-k_{1}}(v_{n+rp-1}\otimes v)
\]
where $w_{i}\in\mathcal{W}_{-i}$, $\deg_{W}w_{i}>0$ and $0<k<rp$ for
$k=k_{1}+\cdots+k_{s}$. From Lemma \ref{teh} and the fact that $w(v_{k}\otimes
x)=v_{k}\otimes wx$ for $w\in\mathcal{W}$, we get $y_{i}\in U(\mathcal{L}%
_{-})_{rp-i}$ such that
\begin{equation}
\sum_{i=0}^{rp-1}y_{i}(v_{n+rp-1-i}\otimes v)=L_{-p}^{r}(v_{n+rp-1}\otimes v).
\label{sing4}%
\end{equation}
Note that $y_{0}=u$. Set $\lambda_{j}:=(n+(r-j)p-1+\alpha+(1-p)\beta)$, and
$\Lambda_{i}=\prod_{j=0}^{r-i-1}\lambda_{j}$. Then (\ref{sing4}) becomes
\begin{equation}
\sum_{i=0}^{rp-1}y_{i}(v_{n+rp-1-i}\otimes v)=\sum_{i=0}^{r-1}(-1)^{r-i}%
\binom{r}{i}\Lambda_{i}v_{n+ip-1}\otimes L_{-p}^{i}v. \label{l p r}%
\end{equation}
Next we want to eliminate components on the right side of (\ref{l p r}),
starting with $v_{n+(r-1)p-1}\otimes L_{-p}^{r-1}v$. Using induction one can
show that
\begin{gather}
-\binom{r}{r-1}\lambda_{0}L_{-p}^{r-1}(v_{n+(r-1)p-1}\otimes v)=\label{b}\\
=-\binom{r}{r-1}\lambda_{0}\sum_{i=0}^{r-1}(-1)^{r-1-i}\binom{r-1}{i}\left(
\prod_{j=1}^{r-1-i}\lambda_{j}\right)  v_{n+ip-1}\otimes L_{-p}^{i}%
v=\nonumber\\
=\sum_{i=0}^{r-1}(-1)^{r-i}\binom{r}{r-1}\binom{r-1}{i}\Lambda_{i}%
v_{n+ip-1}\otimes L_{-p}^{i}v.\nonumber
\end{gather}
The right side of (\ref{l p r})-(\ref{b}) equals
\begin{equation}
\sum_{i=0}^{r-2}(-1)^{r-i}\left(  \binom{r}{i}-\binom{r}{r-1}\binom{r-1}%
{i}\right)  \Lambda_{i}v_{n+ip-1}\otimes L_{-p}^{i}v. \label{b1}%
\end{equation}
Coefficient with $\Lambda_{r-2}$ is $\binom{r}{r-2}-\binom{r}{r-1}\binom
{r-1}{r-2}=-\binom{r}{r-2}$ so we use
\begin{gather}
-\binom{r}{r-2}\lambda_{0}\lambda_{1}L_{-p}^{r-2}(v_{n+(r-2)p-1}\otimes
v)=\label{b2}\\
=-\binom{r}{r-2}\lambda_{0}\lambda_{1}\sum_{i=0}^{r-2}(-1)^{r-2-i}\binom
{r-2}{i}\left(  \prod_{j=2}^{r-1-i}\lambda_{j}\right)  v_{n+ip-1}\otimes
L_{-p}^{i}=\nonumber\\
=-\sum_{i=0}^{r-2}(-1)^{r-i}\binom{r}{r-2}\binom{r-2}{i}\Lambda_{i}%
v_{n-(r-i)p-1}\otimes L_{-p}^{i}\nonumber
\end{gather}
and the right side of (\ref{b1})-(\ref{b2}) is
\[
\sum_{i=0}^{r-3}(-1)^{r-i}\left(  \binom{r}{i}-\binom{r}{r-1}\binom{r-1}%
{i}+\binom{r}{r-2}\binom{r-2}{i}\right)  \Lambda_{i}v_{n+ip-1}\otimes
L_{-p}^{i}v.
\]
In $r-1$ steps we get
\begin{gather*}
(-1)^{r}\binom{r}{0}\left(  \sum_{j=0}^{r-1}(-1)^{j}\binom{r}{r-j}\binom
{r-j}{0}\right)  \Lambda_{0}v_{n-1}\otimes v=-\left(  \prod_{j=0}^{r-1}%
\lambda_{j}\right)  v_{n-1}\otimes v=\\
=-\left(  \prod_{j=0}^{r-1}(n+(r-j)p-1+\alpha+(1-p)\beta)\right)
v_{n-1}\otimes v.
\end{gather*}
Therefore, there exists $x_{i}\in U(\mathcal{L})$ such that
\[
\sum_{i=0}^{rp-1}x_{i}(v_{n+rp-1-i}\otimes v)=\left(  \prod_{j=0}%
^{r-1}(n+(r-j)p-1+\alpha+(1-p)\beta)\right)  v_{n-1}\otimes v.
\]
From condition $\alpha+(1-p)\beta\notin\mathbb{Z}$, follows
\begin{equation}
U_{n-1}\subseteq U_{n}+\cdots+U_{n+rp-1}. \label{c}%
\end{equation}
If $\alpha+2\beta\notin\mathbb{Z}$ we have $U_{n}\supseteq U_{n+1}$ so
(\ref{c}) becomes $U_{n-1}=U_{n}$. Assume $\alpha+2\beta=-k$. Then using
$L_{1}$ and $L_{2}$ we get
\[
\cdots\supseteq U_{k-1}\supseteq U_{k},U_{k+1}\supseteq U_{k+2}\supseteq\cdots
\]
and (\ref{c}) shows $U_{n}\subseteq U_{n+1}$ for $n\neq k-1$, $U_{k-1}%
\subseteq U_{k}+U_{k+1}$. But then
\[
U_{k+2}\subseteq U_{k}\subseteq U_{k+1}=U_{k+2}%
\]
shows $U_{k}=U_{k+1}$ and so $U_{n}=U_{n+1}$ for all $n\in\mathbb{Z}$. This
completes the proof.
\end{proof}

\begin{theorem}
\label{skroz red}Suppose that a singular vector $u^{\prime}\in\mathcal{W}%
\subseteq V(c,h,h_{W})$ generates maximal submodule $J(c,h,h_{W})$. Then the
module $V_{\alpha,\beta,0}^{\prime}\otimes L(c,h,h_{W})$ is reducible.
\end{theorem}

\begin{proof}
Just like in Theorem \ref{ver}, we show that $U_{n}\neq U_{n-1}$. If $x\in
U(\mathcal{L})$ exists, such that $x(v_{n}\otimes v)=v_{n-1}\otimes v$, then
we can find $x_{i}\in U(\mathcal{L})_{-i}$ such that $v_{n-1}\otimes
v=\sum_{i=0}^{k}x_{i+1}(v_{n+i}\otimes v)$. Than $v_{n+k}\otimes x_{k+1}v=0$
so $x_{k+1}v=0$ meaning $x\in U(\mathcal{L})u^{\prime}$. But since $u^{\prime
}\in\mathcal{W}$, we have $x_{k+1}(v_{n+k}\otimes v)=0$ so this component of
the sum is useless. Repeating the process we get a contradiction.
\end{proof}

\begin{theorem}
\label{malo red}Let $u\in V(c,h,h_{W})$ a subsingular vector such that
$\overline{u}=L_{-p}^{r}$. If $\alpha+(1-p)\beta\in\mathbb{Z}$, module
$V_{\alpha,\beta,0}^{\prime}\otimes L(c,h,h_{W})$ is reducible. There exists
$k\in\mathbb{Z}$ such that $U_{k}$ is irreducible.
\end{theorem}

\begin{proof}
Since $\alpha$ is invariant modulo $\mathbb{Z}$ we may assume $\alpha
+(1-p)\beta=0$. From the proof of Theorem \ref{tenz irr} we have
\[
x(v_{n}\otimes v)=-\left(  \prod_{j=0}^{r-1}(n-1+(r-j)p)\right)
v_{n-1}\otimes v
\]
for some $x\in U(\mathcal{L})$, so for $n\notin\left\{  1-p,1-2p,\ldots
,1-rp\right\}  $ we get $U_{n}=U_{n+1}$. Now we prove $U_{1-jp}\neq U_{-jp}$
for $j=1,\ldots,r$. We use the same reasoning as in previous proofs. Assume
otherwise. Then $y\in U(\mathcal{L})$ exists, such that $y(v_{1-jp}\otimes
v)=v_{-jp}\otimes v$, i.e., there are $y_{i}\in U(\mathcal{L}_{-})_{-i}$ such
that
\begin{equation}
v_{-jp}\otimes v=\sum_{i=0}^{m-1}y_{i+1}(v_{1-jp+i}\otimes v). \label{d}%
\end{equation}
But then $y_{m}v=0$. We may assume $y_{m}\in U(\mathcal{L}_{-})u$ because
$u^{\prime}(v_{m-jp}\otimes v)=0$. Let $y_{m}=zu$, $z\in U(\mathcal{L}%
_{-})_{-(m-rp)}$. Then
\[
y_{m}(v_{m-jp}\otimes v)=z\left(  L_{-p}^{r}+\sum w_{i}L_{-k_{s}}\cdots
L_{-k_{1}}\right)  (v_{m-jp}\otimes v).
\]
On the right side we get sum of $z(v_{m-jp-i}\otimes x_{i}v)$ where
$i=0,1,\ldots, rp$. However, since $uv=0$, all the components $v_{m-jp}\otimes
x_{0}v$ add to zero. Like in Lemma \ref{teh} we get $z_{i}\in U(\mathcal{L})$
such that
\[
y_{m}(v_{m-jp}\otimes v)=\sum_{i=1}^{rp}z_{i}(v_{m-i-jp}\otimes v).
\]
However, this means that $y_{m}(v_{m-jp}\otimes v)$ is unnecessary in
(\ref{d}) since it can be expressed with $v_{m-1-jp}\otimes v,\ldots
,v_{m-rp-jp}\otimes v$ (for some $m\geq rp$). Repeating the process we get to
$y_{rp}=u$. However,
\begin{equation}
u(v_{1-jp}\otimes v)=\left(  \sum w_{i}L_{-k_{s}}\cdots L_{-k_{1}}\right)
(v_{1-jp}\otimes v) \label{e}%
\end{equation}
because $L_{-p}^{r}(v_{1-jp}\otimes v)=0$. Obviously, the right side of
(\ref{e}) can not produce $v_{-jp}\otimes v$, leading to contradiction.
Therefore, $U_{1-jp}\subsetneq U_{-jp}$.

Module $U_{1-p}$ is irreducible by Corollary \ref{t}.
\end{proof}

Theorems \ref{ver}, \ref{skroz red}, \ref{malo red} and \ref{tenz irr}
combined prove Theorem \ref{main}.

Note that $U_{1-p}$ in Theorem \ref{malo red} is irreducible module with
infinite-dimensional weight subspaces.

\begin{corollary}
\label{irr sub}Module $V_{\alpha,\beta,0}^{\prime}\otimes L(c,h,h_{W})$
contains an irreducible (not necessarily proper) submodule with
infinite-dimensional weight subspaces if and only if there exists a
subsingular vector $u \in V(c,h,h_{W})$ such that $\overline{u}=L_{-p}^{r}v$.
\end{corollary}

\section{Subquotients of a module $V_{\alpha,\beta,0}^{\prime}\otimes
L(c,h,h_{W})$}

\begin{remark}
\label{kvocijenti}Let $U_{n}/U_{n+1}$ be a nontrivial subquotient in
$V_{\alpha,\beta,0}^{\prime}\otimes L(c,h,h_{W})$. Then
\begin{align*}
L_{k}(v_{n}\otimes v) &  =\lambda v_{n+k}\otimes v\in U_{n+k}\subseteq
U_{n+1},\\
W_{k}(v_{n}\otimes v) &  =0,\quad\text{ for }k>0;\\
L_{0}(v_{n}\otimes v) &  =(h-n-\alpha-\beta)v_{n}\otimes v\\
W_{0}(v_{n}\otimes v) &  =v_{n}\otimes W_{0}v=h_{W}v_{n}\otimes v
\end{align*}
shows that $U_{n}/U_{n+1}$ is highest weight module with the highest weight
$(c,h_{n},h_{W})$, where $h_{n}=h-n-\alpha-\beta$. From Theorem \ref{zd}
follows that $V(c,h_{n},h_{W})$ is reducible if and only if $V(c,h,h_{W})$ is
reducible. If that is the case, formula for a singular vector $u^{\prime}%
\in\mathcal{W}$ is the same in both modules, so $u^{\prime}(v_{n}\otimes
v)=0$. This shows $U_{n}/U_{n+1}$ is isomorphic to either to $L^{\prime
}(c,h_{n},h_{W})$, or to $L(c,h_{n},h_{W})$ (see Corollary \ref{L'}). Of
course, if $h_{n}\neq h_{W}\frac{(13p+1)(p-1)}{12}+\frac{(1-r)p}{2}$ for all
$r\in\mathbb{N}$, then these two modules are the same. Otherwise, the question
is weather $u(v_{n}\otimes v)\in U_{n+1}$.
\end{remark}

Let us review subquotients found in Theorems \ref{ver}, \ref{skroz red}, and
\ref{malo red}.

\begin{theorem}
\label{q1}Verma modules $V(c,h-\alpha-\beta-n,h_{W})$ occur as subquotients in
$V_{\alpha,\beta,0}^{\prime}\otimes V(c,h,h_{W})$ for any $n\in\mathbb{Z}$,
with exception of $n=-\alpha$ if $\alpha\in\mathbb{Z}$ and $\beta=0$, and
$n=-\alpha-1$ if $\alpha\in\mathbb{Z}$ and $\beta=1$.
\end{theorem}

\begin{proof}
For any $n\in\mathbb{Z}$ (with noted exceptions) $U_{n}/U_{n+1}$ is the
highest weight module with the weight $(c,h-\alpha-\beta-n,h_{W})$. (If
$\alpha+2\beta+n=0$ then consider $U_{n}/U_{n+2}$ instead.) Let us prove that
this module is free over $U(\mathcal{L}_{-})$. Note that
\[
U_{n+1}=\operatorname*{span}\left\{  x\left(  v_{n+k}\otimes v\right)  :x\in
U\left(  \mathcal{L}_{-}\right)  ,\text{ }k>0\right\}
\]
and every $x\left(  v_{n+k}\otimes v\right)  $ has a component $v_{n+k}\otimes
xv\neq0$ since Verma module is free over $U\left(  \mathcal{L}_{-}\right)  $.
Suppose $U_{n}/U_{n+1}$ is not free. Then $y\in U(\mathcal{L}_{-})$ exists,
such that $y(v_{n}\otimes v)\in U_{n+1}$. But $y(v_{n}\otimes v)$ can not have
a component $v_{n+k}\otimes xv$ for $k>0$, proving a contradiction.
\end{proof}

\begin{theorem}
\label{q2}If a singular vector $u^{\prime}\in\mathcal{W}$ generates
$J(c,h,h_{W})$, then modules $L(c,h-\alpha-\beta-n,h_{W})$ occur as
subquotients in $V_{\alpha,\beta,0}^{\prime}\otimes L(c,h,h_{W})$ for any
$n\in\mathbb{Z}$, with exception of $n=-\alpha$ if $\alpha\in\mathbb{Z}$ and
$\beta=0$, and $n=-\alpha-1$ if $\alpha\in\mathbb{Z}$ and $\beta=1$.
\end{theorem}

\begin{proof}
Follows from Theorem \ref{skroz red}, Remark \ref{kvocijenti} and the fact
that $L(c,h-\alpha-\beta-n,h_{W})$ is a quotient of $L^{\prime}(c,h-\alpha
-\beta-n,h_{W})$.
\end{proof}

\begin{theorem}
\label{q3}If there is a subsingular vector $u\in V(c,h,h_{W})$ such that
$\overline{u}=L_{-p}^{r}v$ and if $\alpha+\left(  1-p\right)  \beta
\in\mathbb{Z}$ then $L(c,h+(j-\beta)p,h_{W})$ occur as subquotients in
$V_{\alpha,\beta,0}^{\prime}\otimes L(c,h,h_{W})$ for $j=1,\ldots,r$, with the
exception of $L(c,h,0)$ in $V_{0,1,0}^{\prime}\otimes L(c,h,0)$.
\end{theorem}

\begin{proof}
Follows from Remark \ref{kvocijenti} and Theorem \ref{malo red}.
\end{proof}

\begin{remark}
We expect the existence of intertwining operators of type
\[
\binom{L(c,h+(r-\beta)p,h_{W})}{L(c,1-\beta,0)\text{\quad}L(c,h,h_{W})}.
\]
This would result with an elegant proof of Theorems \ref{malo red} and
\ref{q3}.
\end{remark}

\begin{corollary}
$V_{\alpha,\beta,0}^{\prime}\otimes L(c,0,0)$ is irreducible if and only if
$\alpha\notin\mathbb{Z}$. If $2\beta-1\notin\mathbb{N}$ then
\[
\left(  V_{0,\beta,0}^{\prime}\otimes L(c,0,0)\right)  /U_{0}\cong
L(c,1-\beta,0).
\]

\end{corollary}

\begin{proof}
(Sub)singular vector in $V(c,0,0)$ is $L_{-1}v$ so we apply Theorem \ref{q3}.
\end{proof}

Let $2\beta-1=r^{\prime}\in\mathbb{N}$. Then $\left(  V_{0,\beta,0}^{\prime
}\otimes L(c,0,0)\right)  /U_{0}$ is the highest weight module with the
highest weight $(c,\frac{1-r^{\prime}}{2},0)$. Assume there is a subsingular
vector $u\in V(c,\frac{1-r^{\prime}}{2},0)$ (see Conjecture \ref{slut}) such
that $\overline{u}=L_{-1}^{r^{\prime}}v$ and that $L_{k}$ does not occur as a
factor in $u$, for $k>1$ (see Conjecture \ref{slutnja}). Then $u=(L_{-1}%
^{r^{\prime}}+\sum_{i=0}^{r^{\prime}-1}w_{i}L_{-1}^{i})v$ for some $w_{i}%
\in\mathcal{W}$ so
\[
u(v_{-1}\otimes v)=\sum_{i=0}^{r^{\prime}}(-1)^{i}i!v_{-1-i}\otimes w_{i}v.
\]
Since $L_{-1}(v_{0}\otimes v)=W_{-1}(v_{0}\otimes v)=0$, we get
\[
U_{0}=\operatorname*{span}\left\{  x(v_{k}\otimes v):k\in\mathbb{N},x\in
U(\mathcal{L}_{-}\setminus\{L_{-1},W_{-1}\}\right\}  ,
\]
and every $x(v_{k}\otimes v)$ has a component $v_{k}\otimes xv\neq0$. Then,
obviously, $u(v_{-1}\otimes v)\notin U_{0}$ so
\[
\left(  V_{0,\frac{1+r^{\prime}}{2},0}^{\prime}\otimes L(c,0,0)\right)
/U_{0}\cong L^{\prime}(c,\frac{1-r^{\prime}}{2},0).
\]

\begin{remark}
Since intertwining operators of types
\[
\binom{L(c,1-\beta,0)}{L(c,1-\beta,0)\text{\quad}L(c,0,0)},\text{ and }%
\binom{L^{\prime}(c,\frac{1-r^{\prime}}{2},0)}{L^{\prime}(c,\frac{1-r^{\prime
}}{2},0)\text{\quad}L(c,0,0)}%
\]
exist, there are nontrivial $\mathcal{L}$-homomorphisms
\begin{align*}
V_{0,\beta,0}^{\prime}\otimes L(c,0,0)  &  \rightarrow L(c,1-\beta,0),\\
V_{0,\frac{1+r^{\prime}}{2},0}^{\prime}\otimes L(c,0,0)  &  \rightarrow
L^{\prime}(c,\frac{1-r^{\prime}}{2},0).
\end{align*}

\end{remark}

\section{The twisted Heisenberg-Virasoro algebra\label{HV}}

The twisted Heisenberg-Virasoro algebra is the universal central extension of
the Lie algebra of differential operators on a circle of order at most one:
\[
\left\{  f(t)\frac{d}{dt}+g(t):f,g\in\mathbb{C}[t,t^{-1}]\right\}  .
\]
It has an infinite-dimensional Heisenberg subalgebra and a Virasoro
subalgebra. Its highest weight representations have been studied by
E.\ Arbarello et al.\ in \cite{Arbarello} and Billig in \cite{Billig}. In this
paper we focus on zero level case (trivial action of central element of the
Heisenberg subalgebra), studied in \cite{Billig}. These representations occur
in the construction of modules for the toroidal Lie algebras (see
\cite{Billig2}). Our goal is to find irreducible representations with
infinite-dimensional weight spaces. As is case with $W(2,2)$ we study tensor
product of an intermediate series and the highest weight module and use
singular vectors in Verma modules to check irreducibility. However, there are
no subsingular vectors in Verma modules over the Heisenberg-Virasoro algebra.
Still, results are quite similar to those for $W(2,2)$ algebra when Heisenberg
subalgebra acts trivially on an intermediate series. We omit some details
since most proofs are analogous to $W(2,2)$ case.

\bigskip

The twisted Heisenberg-Virasoro algebra $\mathcal{H}$ is a complex Lie algebra
with basis
\[
\{L_{n},I_{n}:n\in\mathbb{Z}\}\cup\{C_{L},C_{LI},C_{I}\}
\]
and Lie bracket
\begin{gather*}
\left[  L_{n},L_{m}\right]  =(n-m)L_{n+m}+\delta_{n,-m}\frac{n^{3}-n}{12}%
C_{L},\\
\left[  L_{n},I_{m}\right]  =-mI_{n+m}-\delta_{n,-m}(n^{2}+n)C_{LI},\\
\left[  I_{n},I_{m}\right]  =n\delta_{n,-m}C_{I},\\
\left[  \mathcal{H},C_{L}\right]  =\left[  \mathcal{H},C_{LI}\right]  =\left[
\mathcal{H},C_{I}\right]  =0.
\end{gather*}
Obviously, $\{L_{n},C_{L}:n\in\mathbb{Z}\}$ spans a Virasoro subalgebra, and
$\{I_{n},C_{I}:n\in\mathbb{Z}\}$ spans a Heisenberg subalgebra. Center of
$\mathcal{H}$ is spanned by $\{I_{0},C_{L},C_{I},C_{LI}\}$ and, unlike $W(0)$
in $W(2,2)$, $I(0)$ acts semisimply on $\mathcal{H}$. Standard $\mathbb{Z}%
$-gradation
\begin{align*}
\mathcal{H}_{n}  &  =\mathbb{C}L_{n}\oplus\mathbb{C}I_{n},\quad n\in
\mathbb{Z}^{\ast},\\
\mathcal{H}_{0}  &  =\mathbb{C}L_{0}\oplus\mathbb{C}I_{0}\oplus\mathbb{C}%
C_{L}\oplus\mathbb{C}C_{I}\oplus\mathbb{C}C_{LI}%
\end{align*}
induces a triangular decomposition
\[
\mathcal{H}=\mathcal{H}_{-}\oplus\mathcal{H}_{0}\oplus\mathcal{H}_{+}%
\]
as usual.

Let $U(\mathcal{H})$ be a universal enveloping algebra. For arbitrary
$h,h_{I},c_{L},c_{I},c_{LI}\in\mathbb{C}$, let $\mathcal{I}$ be a left ideal
in $U(\mathcal{H})$ generated by $\{L_{n},I_{n},L_{0}-h\mathbf{1},I_{0}%
-h_{I}\mathbf{1},C_{L}-c_{L}\mathbf{1},C_{I}-c_{I}\mathbf{1},C_{LI}%
-c_{LI}\mathbf{1}:n\in\mathbb{N}\}$. Then $U(\mathcal{H})/\mathcal{I}$ is a
Verma module denoted with $V(c_{L},c_{I},c_{LI},h,h_{I})$. As usual, it is a
free $U(\mathcal{H}_{-})$-module generated by the highest weight vector
$v:=\mathbf{1}+\mathcal{I}$ and a standard PBW\ basis
\[
\{I_{-m_{s}}\cdots I_{-m_{1}}L_{-n_{t}}\cdots L_{-n_{1}}v:m_{s}\geq\cdots\geq
m_{1}\geq1,n_{t}\geq\cdots\geq n_{1}\geq1\}.
\]
It contains a unique maximal submodule $J(c_{L},c_{I},c_{LI},h,h_{I})$, and
respective quotient module $L(c_{L},c_{I},c_{LI},h,h_{I})=V(c_{L},c_{I}%
,c_{LI},h,h_{I})/J(c_{L},c_{I},c_{LI},h,h_{I}) $ is irreducible. In this paper
we focus on zero level highest weight representations, i.e., $c_{I}=0$ case.

\begin{theorem}
[\cite{Billig}]\label{HV ireducibilnost}Let $c_{I}=0$ and $c_{LI}\neq0$.

\begin{description}
\item[(i)] If $\frac{h_{I}}{c_{LI}}\notin\mathbb{Z}$ or $\frac{h_{I}}{c_{LI}%
}=1$, then Verma module $V(c_{L},0,c_{LI},h,h_{I})$ is irreducible.

\item[(ii)] If $\frac{h_{I}}{c_{LI}}\in\mathbb{Z\setminus\{}1\}$ then
$V(c_{L},0,c_{LI},h,h_{I})$ possesses a singular vector $u\in V_{p}$, where
$p=|\frac{h_{I}}{c_{LI}}-1|$. In this case, a quotient module $L(c_{L}%
,0,c_{LI},h,h_{I})=V(c_{L},0,c_{LI},h,h_{I})/U(\mathcal{H}_{-})u$ is irreducible.
\end{description}
\end{theorem}

Define $I$-degree as follows:
\[
\deg_{I}L_{-n}=0,\text{\quad}\deg_{I}I_{-n}=1,
\]
which induces $\mathbb{Z}$-grading on $U(\mathcal{H})$ and on $V$
\[
\deg_{I}I_{-m_{s}}\cdots I_{-m_{1}}L_{-n_{t}}\cdots L_{-n_{1}}v=s.
\]
For a non-zero $x\in V$ given in standard PBW basis we denote by $\overline
{x}$ its lowest non-zero homogeneous component with respect to $I$-degree.

Also, let $\mathcal{I}=\left\{  I_{-m_{s}}\cdots I_{-m_{1}}v:m_{s}\geq
\cdots\geq m_{1}\geq1\right\}  $. By abuse of notation, we will sometimes
consider $\mathcal{I}$ as a subalgebra of $U(\mathcal{H}_{-})$. We write $V$
short for $V(c_{L},0,c_{LI},h,h_{I})$.

\begin{theorem}
[\cite{Billig}]\label{HV sing}Let $u\in V_{p}$ a singular vector. If
$1-\frac{h_{I}}{c_{LI}}=p\in\mathbb{N}$ then $\overline{u}=L_{-p}v$. If
$\frac{h_{I}}{c_{LI}}-1=p\in\mathbb{N}$ then $\overline{u}=I_{-p}v$ and $u\in
V_{p}\cap\mathcal{I}$.
\end{theorem}

\begin{remark}
If $\frac{h_{I}}{c_{LI}}-1=p\in\mathbb{N}$ one can show that
\begin{equation}
u=w_{0}v+\sum_{i=1}^{p-1}w_{i}L_{-i}v+L_{-p}v, \label{HV s}%
\end{equation}
where $w_{i}\in U\left(  \mathcal{H}_{-}\right)  _{p-i}\cap\mathcal{I}$. (Set
$r=1$ in Lemma \ref{L-rstupanj}.)

For example $(L_{-1}+\frac{h}{c_{LI}}I_{-1})v$ is a singular vector in
$V(c_{L},0,c_{LI},h,0)$, and $I_{-1}v$ in $V(c_{L},0,c_{LI},h,2c_{LI})$.
\end{remark}

\bigskip

Next we define an intermediate series. We take an intermediate series
$\operatorname*{Vir}$-module $V_{\alpha,\beta}$ and let $I_{n}$ act by scalar.
Let $\alpha,\beta,F\in\mathbb{C}$. $V_{\alpha,\beta,F}$ is $\mathcal{H}%
$-module with basis $\{v_{m}:m\in\mathbb{Z}\}$ and action
\begin{align*}
L_{n}v_{m}  &  =-(m+\alpha+\beta+n\beta)v_{m+n},\\
I_{n}v_{m}  &  =Fv_{m+n},\\
C_{L}v_{m}  &  =C_{I}v_{m}=C_{LI}v_{m}=0.
\end{align*}
As with other intermediate series, $V_{\alpha,\beta,F}\cong V_{\alpha
+k,\beta,F}$ for $k\in\mathbb{Z}$. Also, $V_{\alpha,\beta,F}$ is reducible if
and only if $\alpha\in\mathbb{Z}$, $\beta\in\{0,1\}$ and $F=0$. Define
$V_{0,0,0}^{\prime}:=V_{0,0,0}/\mathbb{C}v_{0}$, $V_{0,1,0}^{\prime}%
:=\oplus_{k\neq-1}\mathbb{C}v_{k}$ and $V_{\alpha,\beta,F}^{\prime}%
:=V_{\alpha,\beta,F}$ otherwise. Then $V_{\alpha,\beta,F}^{\prime}$ are
irreducible modules.

It has been shown by R.\ Lu and K.\ Zhao in \cite{Lu-Zhao} that irreducible
$\mathcal{H}$-module with finite-di\-men\-si\-o\-nal weight spaces is either
the highest (or lowest) weight module, or isomorphic to some $V_{\alpha
,\beta,F}^{\prime}$.

\bigskip

Throughout this section we write short $(\overline{c},h,h_{I})$ for
$(c_{L},0,c_{LI},h,h_{I})$. For example, $L(\overline{c},h,h_{I})$ wil denote
the highest weight module $L(c_{L},0,c_{LI},h,h_{I})$.

Now consider module $V_{\alpha,\beta,F}^{\prime}\otimes L(\overline{c}%
,h,h_{I})$. It is generated by $\{v_{m}\otimes v:m\in\mathbb{Z}\}$ where $v$
is the highest weight vector, and has infinite-dimensional weight subspaces.

\begin{theorem}
\label{HV m}Let $\alpha,\beta,F\in\mathbb{C}$ arbitrary. $V_{\alpha,\beta
,F}^{\prime}\otimes L(\overline{c},h,h_{I})$ is irreducible if and only if it
is cyclic on every $v_{m}\otimes v$, $m\in\mathbb{Z}$.
\end{theorem}

\begin{proof}
The proof is similar to $W(2,2)$ case (Theorem \ref{m}). We take nontrivial
submodule $U$ and
\[
x=v_{m-n}\otimes x_{0}+\cdots v_{m}\otimes x_{n}\in U
\]
such that $x_{i}\in L(\overline{c},h,h_{I})_{i}$, and show by induction on $n$
that $v_{k}\otimes v\in U$. If $L_{1}x_{n}\neq0$ or $L_{2}x_{n}\neq0$ we
follow the proof of Theorem 3.2 in \cite{Radobolja2}. Otherwise it must be
$I_{1}x_{n}\neq0$.

If $F=0$, then
\[
I_{1}x=v_{m-n+1}\otimes I_{1}x_{1}+\cdots+v_{m}\otimes I_{1}x_{n}\neq0
\]
and the proof is done. If $F\neq0$ we have
\begin{align*}
(I_{1}^{2}-FI_{2})x  &  =2F\sum_{i=1}^{n}v_{m-n+i+1}\otimes I_{1}x_{i}%
+\sum_{i=2}^{n}v_{m-n+i}\otimes(I_{1}^{2}-FI_{2})x_{i}=\\
&  =Fv_{m+1}\otimes I_{1}x_{n}+\sum_{i=0}^{n-2}v_{m-n+i+2}\otimes
(2FI_{1}x_{i+1}+(I_{1}^{2}-FI_{2})x_{i+2})
\end{align*}
so again, there is a vector in $U$ with less than $n+1$ components, and by
induction, there is some $v_{k}\otimes v\in U$.
\end{proof}

We denotation by $U_{k}$ submodule $U(\mathcal{H})(v_{k}\otimes v)$ for any
$k\in\mathbb{Z}$. In order to prove irreducibility of $V_{\alpha,\beta
,F}^{\prime}\otimes L(\overline{c},h,h_{I})$, it suffices to show
$U_{n}=U_{n+1}$ for all $n$.

\begin{theorem}
Module $V_{\alpha,\beta,F}^{\prime}\otimes V(\overline{c},h,h_{I})$ is
reducible. Modules $V(\overline{c},h-\alpha-\beta-n,h_{I}+F)$ occur as
subquotients, for any $n\in\mathbb{Z}$, with exception of $n=-\alpha$ if
$\alpha\in\mathbb{Z}$ and $\beta=0$, and $n=-\alpha-1$ if $\alpha\in
\mathbb{Z}$ and $\beta=1$.
\end{theorem}

\begin{proof}
Analogous to Theorems \ref{ver} and \ref{q1}.
\end{proof}

\bigskip

First we consider case $F=0$.

\begin{theorem}
Let $1-\frac{h_{I}}{c_{LI}}=p\in\mathbb{N}$. Then $V_{\alpha,\beta,0}^{\prime
}\otimes L(\overline{c},h,h_{I})$ is reducible if and only if $\alpha
+(1-p)\beta\in\mathbb{Z}$. In that case $(V_{\alpha,\beta,0}^{\prime}\otimes
L(\overline{c},h,h_{I}))/U_{1-p}$ is the highest weight module with the
highest weight $(\overline{c},h+p(1-\beta),h_{I})$.
\end{theorem}

\begin{proof}
The proof of irreducibility is analogous to that of Theorem \ref{tenz irr} for
$r=1$, since we have a singular vector (\ref{HV s}). Proof of reducibility is
analogous to Theorem \ref{q3} for $r=1$.
\end{proof}

\begin{theorem}
Let $\frac{h_{I}}{c_{LI}}-1\in\mathbb{N}$. Then $V_{\alpha,\beta,0}^{\prime
}\otimes L(\overline{c},h,h_{I})$ is reducible. Subquotients $U_{n}/U_{n+1}$
are isomorphic to $L(\overline{c},h-\alpha-\beta-n,h_{I})$.
\end{theorem}

\begin{proof}
Analogous to Theorem \ref{q2}.
\end{proof}

\bigskip

Now we give a general result for arbitrary $F$, analogous to Theorem 5 in
\cite{Zhang}.

\begin{theorem}
Let $\left\vert 1-\frac{h_{I}}{c_{LI}}\right\vert =p\in\mathbb{N}$, and
$(h,h_{I})\neq(0,0)$. If $F$ is transcendental over $\mathbb{Q}(\alpha
,\beta,c_{L},c_{LI},h,h_{I})$ or $F$ is algebraic over $\mathbb{Q}%
(\alpha,\beta,c_{L},c_{LI},h,h_{I})$ with degree greater than $p$, then
$V_{\alpha,\beta,F}^{\prime}\otimes L(\overline{c},h,h_{I})$ is irreducible.
\end{theorem}

\begin{proof}
Since $F\neq0$ we have
\[
I_{1}(v_{n-1}\otimes v)=Fv_{n}\otimes v\neq0
\]
so $U_{n}\subseteq U_{n-1}$.

Let $p=\frac{h_{I}}{c_{LI}}-1$. Then $u\in\mathcal{I}$. Since
\[
I_{-j_{t}}\cdots I_{-j_{1}}(v_{n+p-1}\otimes v)=F^{t}v_{n-1}\otimes
v+\cdots+v_{n+p-1}\otimes I_{-j_{t}}\cdots I_{-j_{1}}v
\]
for $j_{1}+\cdots+j_{t}=p$, we find $u^{\prime}\in U(\mathcal{H})$ such that
\[
u^{\prime}(v_{n}\otimes v)=Fs(F)v_{n-1}\otimes v
\]
for some $s(F)\in\mathbb{Q}(h_{I},c_{LI})\left[  F\right]  $, $\deg s=p-1$. By
assumption, $Fs(F)\neq0$ so $U_{n-1}\subseteq U_{n}$.

Now let $p=1-\frac{h_{I}}{c_{LI}}$. Then we apply (\ref{HV s}). Since
\[
I_{-j_{t}}\cdots I_{-j_{1}}L_{-i}(v_{n+p}\otimes v)=-F^{t}(n+p+\alpha
+\beta-i\beta)v_{n}\otimes v+
\]%
\begin{align*}
&  +\sum_{k}F^{k}(n+p+\alpha+\beta-i\beta)\sum_{j}v_{n+p-j}\otimes
x_{j}^{(t-k)}v+\\
&  +\sum_{l}F^{l}\sum_{j}v_{n+p-j-i}\otimes y_{j}^{(t-l)}L_{-i}v+v_{n+p}%
\otimes I_{-j_{t}}\cdots I_{-j_{1}}L_{-i}v
\end{align*}
where $j_{1}+\cdots+j_{t}=p-i$ and $\deg_{I}x_{j}^{(r)}=\deg_{I}y_{j}^{(r)}=r
$, we get
\[
u(v_{n+p}\otimes v)=f(F)v_{n}\otimes v+\sum_{i=1}^{p-1}v_{n+i}\otimes z_{i}v
\]
for some $z_{i}\in U(\mathcal{H}_{-})_{-i}$ and polynomial $f(F)$. Then we
find $u^{\prime}\in U(\mathcal{H})$ such that
\[
u^{\prime}(v_{n+1}\otimes v)=\left(  \strut q(F)n+r(F)\right)  v_{n}\otimes
v,
\]
where $q(F),r(F)\in\mathbb{Q}(\alpha,\beta,h,h_{I},c_{L},c_{LI})\left[
F\right]  $, $\deg q=p-1$, $\deg r=p$. Again, this shows $U_{n-1}\subseteq
U_{n}$.
\end{proof}

\begin{theorem}
$V_{\alpha,\beta,F}^{\prime}\otimes L(\overline{c},0,0)$ is irreducible if and
only if $\alpha\notin\mathbb{Z}$. Moreover,
\begin{align*}
(V_{0,\beta,F}^{\prime}\otimes L(\overline{c},0,0))/U_{0}  &  \cong
V(\overline{c},1-\beta,F)\text{, if }\beta\neq1\text{,}\\
(V_{0,1,F}^{\prime}\otimes L(\overline{c},0,0))/U_{0}  &  \cong V(\overline
{c},1,F).
\end{align*}

\end{theorem}

\begin{proof}
Let $\alpha\notin\mathbb{Z}$. Since $L_{-1}v=0$ in $L(\overline{c},0,0)$ we
have
\[
L_{-1}(v_{n}\otimes v)=-(n+\alpha)v_{n-1}\otimes v
\]
which prooves $U_{n}=U_{n-1}$ for all $n$, so $V_{\alpha,\beta,F}^{\prime
}\otimes L(\overline{c},0,0)$ is irreducible.

Now let $\alpha=0$. We may set $p=r=1$ in the proof of Theorem \ref{malo red}
to show reducibility of $U_{-1}=V_{\alpha,\beta,F}^{\prime}\otimes
L(\overline{c},0,0)$. In particular, $U_{0}$ is an irreducible submodule.
$U_{-1}/U_{0}$ is the highest weight module with the highest weight
$(\overline{c},1-\beta,F)$ if $\beta\neq1$, and $U_{-2}/U_{0}$ is the highest
weight module of the highest weight $(\overline{c},1,F)$ if $\beta=1$. It is
left to prove that quotient module is free over $U(\mathcal{H}_{-})$. Since
$U_{0}$ is spanned with vectors $x(v_{k}\otimes v)$, where $x\in
U(\mathcal{H}_{-}\setminus\{L_{-1}\})$, $k\in\mathbb{Z}_{+}$, it follows that
every vector in $U_{0}$ has a component $v_{k}\otimes xv$ for some $k\geq0$.
Therefore, $y(v_{-1}\otimes v)\notin U_{0}$ for any $y\in U(\mathcal{H}_{-})$,
so $U_{-1}/U_{0}$ is free, i.e., a Verma module.
\end{proof}

\bigskip

\bigskip

\begin{acknowledgement}
I would like to thank my advisor Professor Dra\v{z}en Adamo\-vi\'{c} for his
ideas, guidance and patience. Some of these results were presented at the
conference Representation Theory XII in Dubrovnik in 2011.\ and are a part of
author's Ph.D.\ dissertation \cite{Radobolja1} written under the direction of
Professor D.\ Adamovi\'{c}.
\end{acknowledgement}


\begin{thebibliography}{99}                                                                                               %


\bibitem {Adamovic1}Adamovi\'{c}, D., "New irreducible modules for affine Lie
algebras at the critical level", International Math.\ Research Notices 6,
253-262 (1996).

\bibitem {Adamovic2}Adamovi\'{c}, D., "Vertex operator algebras and
irreducibility of certain modules for affine Lie algebras", Math.\ Research
Letters 4, 809-821 (1997).

\bibitem {Adamovic3}Adamovi\'{c}, D., "An application of $U(g)$-bimodules to
representation theory of affine Lie algebras", Algebras and Rep.\ theory 7, 4,
457-469 (2004).

\bibitem {Arbarello}Arbarello, E., De Concini, C., Kac, V.G. and Procesi, C.,
"Moduli spaces of curves and representation theory", Comm.\ Math.\ Phys. 117,
1-36 (1988).

\bibitem {Billig}Billig, Y., "Representations of the twisted
Heisenberg-Virasoro algebra at level zero", Canadian Math.\ Bulletin 46,
529-537 (2003).

\bibitem {Billig2}Billig, Y., "Energy-momentum tensor for the toroidal Lie
algebras", J.\ Algebra 308 1, 252-269 (2007).

\bibitem {Chari-Pressley}Chari, V. and Pressley, A., "A New Family of
Irreducible, Integrable Modules for Affine Lie Algebras", Math.\ Ann., 543-562\ (1987).

\bibitem {Frenkel - Huang - Lepowsky}Frenkel, I.\ B., Huang, Y. and Lepowsky,
J., "On Axiomatic Approaches to Vertex Operator Algebras and Modules",
Mem.\ Amer.\ Math.\ Soc.\ 104 (1993).

\bibitem {Frenkel - Lepowsky - Meurman}Frenkel, I., Lepowsky, J. and Meurman,
A., "Vertex Operator Algebras and the Monster", Pure and Appl.\ Math., 134 (1988).

\bibitem {Jiang-Pei}Jiang, W. and Pei, Y., "On the structure of Verma modules
over the $W$-algebra $W(2,2)$", J.\ Math.\ Phys.\ 51, 022303 (2010).

\bibitem {Lepowsky-Li}Lepowsky, J. and Li, H., \textit{Introduction to Vertex
Operator Algebras and Their Representations} (Birkh\"{a}user, Basel, 2004).

\bibitem {Lu-Zhao}Lu, R. and Zhao, K., "Classification of irreducible weight
modules over the twisted Heisenberg-Virasoro algebra",
Comm.\ Cont.\ Math.\ 12, 2, 183-205 (2010).

\bibitem {Liu-Zhu}Liu, D. and Zhu, L., "Classification of Harish Chandra
modules over the $W$-algebra $W(2,2)$", J.\ Math.\ Phys.\ 49, 012901 (2008).

\bibitem {Radobolja1}Radobolja, G., Ph.\ D.\ thesis (in Croatian), University
of Zagreb, 2012. http://bib.irb.hr/prikazi-rad?\&rad=605718

\bibitem {Radobolja2}Radobolja, G., "Application of vertex algebras to the
structure theory of certain representations over Virasoro algebra", Algebras
and Represent.\ Theory 16 (2013)

\bibitem {Zhang-Dong}Zhang, W.\ and Dong, C., "$W$-algebra $W\left(
2,2\right)  $ and the vertex operator algebra $L\left(  \frac{1}{2},0\right)
\otimes L\left(  \frac{1}{2},0\right)  $", Commun.\ Math.\ Phys.\ 285,
991-1004 (2009).

\bibitem {Zhang}Zhang, H., "A class of representations over the Virasoro
Algebra", J. Algebra 190, 1-10 (1997).
\end{thebibliography}
\end{document}